\documentclass[12pt,draft]{amsart}
\usepackage{amsmath,amssymb,amsthm,bm}
\usepackage{graphicx,enumerate}
\usepackage{layout}

\theoremstyle{plain}
\newtheorem{theorem}{Theorem}[section]
\newtheorem{lemma}[theorem]{Lemma}
\newtheorem{proposition}[theorem]{Proposition}
\newtheorem{corollary}[theorem]{Corollary}

\newtheorem*{subfield-problem}{Subfield problem of a generic polynomial}
\newtheorem*{isomorphism-problem}{Field isomorphism problem of a generic polynomial}
\newtheorem*{theorem-nonumber}{Theorem}
\newtheorem*{corollary-nonumber}{Corollary}

\theoremstyle{definition}
\newtheorem*{definition}{Definition}
\newtheorem{example}[theorem]{Example}
\newtheorem{remark}[theorem]{Remark}
\newtheorem*{acknowledgment}{Acknowledgment}

\newcommand{\bs}{\mathbf{s}}\newcommand{\bt}{\mathbf{t}}
\newcommand{\ba}{\mathbf{a}}\newcommand{\bb}{\mathbf{b}}
\newcommand{\bx}{\mathbf{x}}\newcommand{\by}{\mathbf{y}}
\newcommand{\As}{A_\mathbf{s}}\newcommand{\Bs}{B_\mathbf{s}}
\newcommand{\Cs}{C_\mathbf{s}}\newcommand{\Ds}{D_\mathbf{s}}
\newcommand{\At}{A_\mathbf{t}}\newcommand{\Bt}{B_\mathbf{t}}
\newcommand{\Dt}{D_\mathbf{t}}
\newcommand{\Aa}{A_\mathbf{a}}\newcommand{\Ba}{B_\mathbf{a}}
\newcommand{\Ca}{C_\mathbf{a}}\newcommand{\Da}{D_\mathbf{a}}
\newcommand{\Ab}{A_\mathbf{b}}\newcommand{\Bb}{B_\mathbf{b}}
\newcommand{\Db}{D_\mathbf{b}}
\newcommand{\Dels}{\Delta_\mathbf{s}}\newcommand{\Delt}{\Delta_\mathbf{t}}
\newcommand{\Dela}{\Delta_\mathbf{a}}\newcommand{\Delb}{\Delta_\mathbf{b}}
\newcommand{\betas}{\beta_\mathbf{s}}\newcommand{\betat}{\beta_\mathbf{t}}
\newcommand{\betaa}{\beta_\mathbf{a}}\newcommand{\Es}{E_\mathbf{s}}
\newcommand{\Gs}{G_\mathbf{s}}\newcommand{\Gt}{G_\mathbf{t}}
\newcommand{\Gst}{G_{\mathbf{s},\mathbf{t}}}

\newcommand{\Gab}{G_{\mathbf{a},\mathbf{b}}}

\title[Geometric framework for the subfield problem]{A geometric framework for 
the subfield problem of generic polynomials via Tschirnhausen transformation}

\thanks{This work was partially supported by Grant-in-Aid for Scientific 
Research (C) 19540057 of Japan Society for the Promotion of Science}

\author{Akinari Hoshi and Katsuya Miyake}


\begin{document}
\maketitle
\begin{abstract}
Let $k$ be an arbitrary field. 
We study a general method to solve the subfield problem of generic polynomials for the 
symmetric groups over $k$ via Tschirnhausen transformation. 
Based on the general result in the former part, we give an explicit solution to 
the field isomorphism problem and the subfield problem of cubic generic polynomials 
for $\frak{S}_3$ and $C_3$ over $k$. 
As an application of the cubic case, we also give several sextic generic polynomials over $k$. 
\end{abstract}
\section{Introduction}\label{seIntro}
Let $k$ be a fixed base field of arbitrary characteristic and $G$ a finite group. 
Let $k(\bt)$ be the rational function field over $k$ with $m$ variables $\bt=(t_1,\ldots,t_m)$. 
A polynomial $F(t_1,\ldots,t_m;X)\in k(\bt)[X]$ is called $k$-generic for $G$ if the Galois 
group of $F(\bt;X)$ over $k(\bt)$ is isomorphic to $G$ and every $G$-Galois extension $L/M$ 
with $M\supset k$ and $\# M=\infty$ can be obtained as $L=\mathrm{Spl}_M F(\ba;X)$, 
the splitting field of $F(\ba;X)$ over $M$, for some $\ba=(a_1,\ldots,a_m)\in M^m$. 
By Kemper's Theorem \cite{Kem01}, furthermore, every $H$-Galois extension for a subgroup 
$H$ of $G$ over an infinite field $M$ is also given by a specialization of $F(\bt;X)$. 
Examples of generic polynomials for various $G$ are found, for example, in \cite{JLY02}. 
We also give several sextic $k$-generic polynomials in Section 6. 
The fact that a $k$-generic polynomial for $G$ covers all $H$-Galois extensions ($H\subset G$) 
over $M\supset k$ by specializing parameters naturally raises a problem; namely, 
\begin{subfield-problem}
Let $F(\bt;X)$ be a $k$-generic polynomial for $G$. 
For a field $M\supset k$ and $\ba,\bb\in M^m$, determine whether $\mathrm{Spl}_M F(\bb;X)$ 
is a subfield of $\mathrm{Spl}_M F(\ba;X)$ or not.
\end{subfield-problem}
When we restrict ourselves to $G$-Galois extensions over $M$ for a fixed group $G$, 
we are to consider a special case of the subfield problem: 
\begin{isomorphism-problem}
Determine whether $\mathrm{Spl}_M F(\ba;X)$ and 
$\mathrm{Spl}_M F(\bb;X)$ are isomorphic over $M$ or not 
for $\ba,\bb\in M^m$. 
\end{isomorphism-problem}

In this paper, we develop a method to solve the field isomorphism problem of 
$k$-generic polynomial for the symmetric group $\frak{S}_n$ of degree $n$ 
via Tschirnhausen transformation. 

In Section 2, we investigate a geometric interpretation of Tschirnhausen transformations. 
Our main idea, which we describe here briefly, is as follows: 
Let $f(X)$ and $g(X)$ be monic separable polynomials of degree $n$ in $k[X]$ for a field $k$ 
with roots $\alpha_1,\ldots,\alpha_n$ and $\beta_1,\ldots,\beta_n$, respectively, 
in a fixed algebraic closure of $k$. 
A polynomial $g(X)\in k[X]$ is called a Tschirnhausen transformation of $f(X)$ over 
$M$ $(\supset k)$ if $g(X)$ is of the form 
\begin{align}
g(X)=\prod_{i=1}^n 
\bigl(X-(c_0+c_1\alpha_i+\cdots+c_{n-1}\alpha_i^{n-1})\bigr),\quad c_i \in M.\label{eqci}
\end{align}

Two polynomials $f(X)$ and $g(X)$ in $k[X]$ are Tschirnhausen equivalent over $M$ 
if they are Tschirnhausen transformations over $M$ of each other. 
For two irreducible separable polynomials $f(X)$, $g(X)\in k[X]$, the following two 
conditions are equivalent: 
(i) $f(X)$ and $g(X)$ are Tschirnhausen equivalent over $M$; 
(ii) the quotient fields $M[X]/(f(X))$ and $M[X]/(g(X))$ are isomorphic over $M$. 

Now we replace the roots $\bm{\alpha}=(\alpha_1,\ldots,\alpha_n)$ and 
$\bm{\beta}=(\beta_1,\ldots,\beta_n)$ by independent variables 
$\bx=(x_1,\ldots,x_n)$ and $\by=(y_1,\ldots,y_n)$ respectively. 
If we take $u_i:=u_i(\bx,\by)\in M[\bx,\by,\Dels^{-1}]$ as 
\begin{align*}
u_i(\mathbf{x},\mathbf{y})=\Dels^{-1}\cdot\mathrm{det}
\left(\begin{array}{cccccccc}
1 & x_1 & \cdots & x_1^{i-1} & y_1 & x_1^{i+1} & \cdots & x_1^{n-1}\\ 
1 & x_2 & \cdots & x_2^{i-1} & y_2 & x_2^{i+1} & \cdots & x_2^{n-1}\\
\vdots & \vdots & & \vdots & \vdots & \vdots & & \vdots\\
1 & x_n & \cdots & x_n^{i-1} & y_n & x_n^{i+1} & \cdots & x_n^{n-1}
\end{array}\right)
\end{align*}
where $\Dels=\prod_{1\leq i<j\leq n} (x_j-x_i)$, then $u_0,\ldots,u_{n-1}$ 
correspond to the coefficients $c_0,\ldots,c_{n-1}$ of a Tschirnhausen transformation from 
$f(X)$ to $g(X)$ as in (\ref{eqci}). 
Let $\frak{S}_n\times \frak{S}_n$ act on $\{\bx,\by\}$ by permuting variables. 
There are $n!$ conjugates of $u_i$ under the action of $\frak{S}_n\times \frak{S}_n$ because 
the stabilizer $H:=\mathrm{Stab}_{\frak{S}_n\times \frak{S}_n}(u_i)$ of $u_i$ in 
$\frak{S}_n\times \frak{S}_n$ is the diagonal subgroup of $\frak{S}_n\times \frak{S}_n$ 
with $H\cong \frak{S}_n$. 
Let $s_i$ (resp. $t_i$) be the $i$-th elementary symmetric function in $n$ 
variables $\bx$ (resp. $\by$). We put $K:=k(\bs,\bt)=k(s_1,\ldots,s_n,t_1,\ldots,t_n)$, 
$L_\bs:=K(\bx)$ and $L_\bt:=K(\by)$. 
Then we have, for each $i$, $0\leq i\leq n-1$, the following basic properties:
\begin{align*}
&{\rm (i)}\quad \ \ 
(L_\bs L_\bt)^{g^{-1}Hg}=K(u_0^g,\ldots,u_{n-1}^g)=K(u_i^g)\ \ \textrm{for}\ \ g\in \Gst,\\
&{\rm (ii)}\quad \ \, L_\bs\cap K(u_i^g)=L_\bt\cap K(u_i^g)=K\ \ \textrm{for}\ \ g\in \Gst,\\
&{\rm (iii)}\quad \ L_\bs L_\bt=L_\bs(u_i^g)=L_\bt(u_i^g)\ \ \textrm{for}\ \ g\in \Gst,\\
&{\rm (iv)}\quad \ L_\bs L_\bt=K(u_i^g\ |\ \overline{g}\in H\backslash \Gst).
\end{align*}

In Section 3, we study a specialization of parameters 
$(\bs,\bt)\mapsto (\ba,\bb)\in M^n\times M^n$ of polynomials 
$f_n(\bs;X)=\prod_{i=1}^n(X-x_i)$ and $f_n(\bt;X)=\prod_{i=1}^n(X-y_i)$. 
We always assume that such a specialization $(\bs,\bt)\mapsto (\ba,\bb)\in M^n\times M^n$ 
satisfy the condition $\Dela\cdot\Delb\neq 0$, i.e. both of $f_n(\ba;X)\in M[X]$ and 
$f_n(\bb;X)\in M[X]$ are separable. 
For $i$, $0\leq i\leq n-1$, we take a polynomial 
\begin{align*}
F_i(\ba,\bb;X)\, 
&:=\, \prod_{\overline{g}\in H\backslash \frak{S}_n\times\frak{S}_n} 
(X-u_i^g(\bm{\alpha},\bm{\beta}))
\, =\, \prod_{\overline{g}\in H\backslash \frak{S}_n\times\frak{S}_n} 
(X-c_i^g)\in M[X],
\end{align*}
of degree $n!$. 
Therefore we obtain the following: 
\begin{theorem-nonumber}[Theorem \ref{throotf}]
For a fixed $j, 0\leq j \leq n-1$, and $\ba,\bb \in M^n$ with 
$\Delta_\ba\cdot\Delta_\bb\neq 0$, assume that the polynomial $F_j(\ba,\bb;X)$ has no 
multiple root. Then $F_j(\ba,\bb;X)$ has a root in $M$ if and only if $M[X]/(f_n(\ba;X))$ 
and $M[X]/(f_n(\bb;X))$ are $M$-isomorphic. 
\end{theorem-nonumber}
\begin{corollary-nonumber}[Corollary \ref{cor1}]
Let $j$ and $\ba,\bb \in M^n$ be as above. 
Assume that both of $\mathrm{Gal}(f_n(\ba;X)/M)$ and $\mathrm{Gal}(f_n(\bb;X)/M)$ 
are isomorphic to a transitive subgroup $G$ of $\frak{S}_n$ and that all subgroups of $G$ 
with index $n$ are conjugate in $G$. 
Then $F_j(\ba,\bb;X)$ has a root in $M$ if and only if 
$\mathrm{Spl}_M f_n(\ba;X)$ and $\mathrm{Spl}_M f_n(\bb;X)$ coincide. 
\end{corollary-nonumber}

In Sections 4 and 5, based on the general result in Sections 2 and 3, we give an explicit 
solution to the field isomorphism problem and the subfield problem of $k$-generic polynomials 
for $\frak{S}_3$ and for $C_3$. 
Here we display some results to $k$-generic polynomials $g^{S_3}(s;X):=X^3+sX+s$ for $S_3$ 
and $g^{C_3}(s;X):=X^3-sX^2-(s+3)X-1$ for $C_3$. 

For $g^{S_3}(s;X)$, we take coefficients $c_0,c_1,c_2$ of a Tschirnhausen transformation 
from $g^{S_3}(a;X)$ to $g^{S_3}(b;X)$ and we put $u:=3c_1/c_2$. 
Then we get $H(a,b;X):=(a-b)\cdot \prod_{\overline{g}\in H\backslash 
\frak{S}_n\times\frak{S}_n}(X-u^g)$ where 
\[
H(a,b;X)=a(X^2+9X-3a)^3-b(X^3-2aX^2-9aX-2a^2-27a)^2.
\]
\begin{theorem-nonumber}[Theorem \ref{thS3}]
Assume that char $k\neq 3$. 
For $a,b\in M$ with $a\neq b$, the decomposition type 
of irreducible factors $h_{\mu}(X)$ of $H(a,b;X)$ over $M$ gives 
an answer to the subfield problem of $X^3+sX+s$ as on Table $1$ in Section 4. 
In particular, two splitting fields of $X^3+aX+a$ and of $X^3+bX+b$ 
over $M$ coincide if and only if there exists $u\in M$ such that 
\[
b=\frac{a(u^2+9u-3a)^3}{(u^3-2au^2-9au-2a^2-27a)^2}.
\]
\end{theorem-nonumber}

For $g^{C_3}(s;X)$, we obtain the following theorem which is an analogue 
to the results of Morton \cite{Mor94} and Chapman \cite{Cha96}. 
\begin{theorem-nonumber}[Theorem \ref{thC3}]
Assume that char $k\neq 2$. 
For $m,n\in M$, two splitting fields of $X^3-mX^2-(m+3)X-1$ and of $X^3-nX^2-(n+3)X-1$ 
over $M$ coincide if and only if there exists $z\in M$ such that either 
\begin{align*}
n\,=\,\frac{m(z^3-3z-1)-9z(z+1)}{mz(z+1)+z^3+3z^2-1}\ \mathit{or}\ 
n\,=\,-\frac{m(z^3+3z^2-1)+3(z^3-3z-1)}{mz(z+1)+z^3+3z^2-1}.
\end{align*}
\end{theorem-nonumber}
By applying Hilbert's irreducibility theorem (cf. for example \cite[Chapter 3]{JLY02}) 
and Siegel's theorem for curves of genus $0$ (cf. \cite[Theorem 6.1]{Lan78}, 
\cite[Chapter 8, Section 5]{Lan83}, \cite[Theorem D.8.4]{HS00}) 
to the preceding theorems respectively, we get the following corollaries: 
\begin{corollary-nonumber}[Corollary \ref{Hil1} and Corollary \ref{Hil2}]
Let $g^G(a;X)=X^3+aX+a$ $($resp. $X^3-aX^2-(a+3)X-1)$ be as above 
with given $a\in M$, and suppose that $M\supset k$ is Hilbertian 
$($e.g. a number field\,$)$. 
Then there exist infinitely many $b\in M$ such that 
$\mathrm{Spl}_{M} g^G(a;X)=\mathrm{Spl}_M g^G(b;X)$. 
\end{corollary-nonumber}
\begin{corollary-nonumber}[Corollary \ref{Siegel1} and Corollary \ref{Siegel2}]
Let $M$ be a number field and $\mathcal{O}_M$ the ring of integers in $M$. 
For $g^G(a;X)=X^3+aX+a$ $($resp. $X^3-aX^2-(a+3)X-1)$ as above with a given integer 
$a\in \mathcal{O}_M$, there exist only finitely many integers $b\in\mathcal{O}_M$ such that 
$\mathrm{Spl}_{M} g^G(a;X)=\mathrm{Spl}_M g^G(b;X)$. 
\end{corollary-nonumber}
In Section 6, as an application of the cubic case, 
we also give several sextic $k$-generic polynomials. 

The calculations in this paper were carried out with Mathematica \cite{Wol03}.

\section{Tschirnhausen transformation (geometric interpretation)}
\label{seTschirn}

Let $n\geq 3$ be a positive integer. 
Let $\mathbf{x}=(x_1,\ldots,x_n)$ and $\mathbf{y}=(y_1,\ldots,y_n)$ be 
$2n$ independent variables over $k$. 
Let $f_n(\bs;X)=f_n(s_1,\ldots,s_n;X) \in k(\bs)[X]$ 
(resp. $f_n(\bt;X)=f_n(t_1,\ldots,t_n;X) \in k(\bt)[X]$) be a monic polynomial 
of degree $n$ whose roots are $x_1,\ldots,x_n$ (resp. $y_1,\ldots,y_n$). 
Then we have 
\begin{align*}
f_n(\bs;X)\, &=\, \prod_{i=1}^n(X-x_i)
\, =\, X^n-s_1X^{n-1}+s_2X^{n-2}+\cdots+(-1)^n s_n,\\
f_n(\bt;X)\, &=\, \prod_{i=1}^n(X-y_i)
\, =\, X^n-t_1X^{n-1}+t_2X^{n-2}+\cdots+(-1)^n t_n,
\end{align*}
where \,$s_i$\, (resp. $t_i$)\, is the $i$-th elementary symmetric function in 
$n$ variables $x_1,\ldots,x_n$ (resp. $y_1,\ldots,y_n$). 
We put 
\[
K\ :=\ k(\bs,\bt); 
\]
it is naturally regarded as the rational function field over $k$ with $2n$ 
variables. We put 
\begin{align*}
L_\bs\ &:=\ \mathrm{Spl}_K\, f_n(\bs;X)\, =\, K(x_1,\ldots,x_n),\\
L_\bt\ &:=\ \mathrm{Spl}_K\, f_n(\bt;X)\, =\, K(y_1,\ldots,y_n).
\end{align*}
Then we have $L_\bs\, \cap\, L_\bt=K$ and $L_\bs L_\bt=k(\mathbf{x},\mathbf{y})$. 
The field extension $k(\mathbf{x},\mathbf{y})/K$ is a Galois extension 
whose Galois group is isomorphic to $\frak{S}_n\times \frak{S}_n$, 
where $\frak{S}_n$ is the symmetric group of degree $n$. 
Put 
\begin{align*}
\Gs\ :=\ \mathrm{Gal}(L_\bs L_\bt/L_\bt),\ \ 
\Gt\ :=\ \mathrm{Gal}(L_\bs L_\bt/L_\bs) 
\end{align*}
and
\[
\Gst\ :=\ \Gs\times\Gt. 
\]
Then we have 
$\Gst\cong \mathrm{Gal}(L_\bs L_\bt/K)$ which acts on 
$k(\mathbf{x},\mathbf{y})$ from the right. 
For $g=(\sigma,\tau)\in \Gst$, we take anti-isomorphisms 
\begin{align*}
\varphi\ \ &{\rm :}\ \ \Gs\rightarrow \frak{S}_n,\ \ \sigma\mapsto 
\varphi(\sigma),\\ 
\psi\ \ &{\rm :}\ \ \Gt\rightarrow \frak{S}_n,\ \ \tau\mapsto \psi(\tau), 
\end{align*}
and we regard the group $\Gst$ as $\frak{S}_n\times \frak{S}_n$ 
by the rule 
\begin{align*}
x_i^\sigma=x_{\varphi(\sigma)(i)},\quad y_i^\sigma=y_i,\quad 
x_i^\tau=x_{i},\quad y_i^\tau=y_{\psi(\tau)(i)},\quad (i=1,\ldots,n). 
\end{align*}
In a fixed algebraic closure of $K$, there exist $n!$ Tschirnhausen 
transformations from $f_n(\bs;X)$ to $f_n(\bt;X)$. 
We first study the field of definition of Tschirnhausen transformations 
from $f_n(\bs;X)$ to $f_n(\bt;X)$. 
Let 
\begin{align*}
D:=
\left(
\begin{array}{ccccc}
1 & x_1 & x_1^2 & \cdots & x_1^{n-1}\\ 
1 & x_2 & x_2^2 & \cdots & x_2^{n-1}\\
\vdots & \vdots & \vdots & \ddots & \vdots\\
1 & x_n & x_n^2 & \cdots & x_n^{n-1}\end{array}\right)
\end{align*}
be a so-called Vandermonde matrix of size $n$. 
The matrix $D$ is invertible because 
\[
{\rm det}\, D\, =\, \Dels,\ \ 
\mathrm{where}\ \ \Dels\, :=\prod_{1\leq i<j\leq n} (x_j-x_i). 
\]
Note that $k(\bs)(\Dels)$ is a quadratic extension of $k(\bs)$ when char $k\neq 2$. 
We define the $n$-tuple $(u_0(\mathbf{x},\mathbf{y}),\ldots,
u_{n-1}(\mathbf{x},\mathbf{y}))\in k[\mathbf{x},\mathbf{y},\Dels^{-1}]^n$ by
\begin{align}
\left(\begin{array}{c}u_0(\mathbf{x},\mathbf{y})\\ u_1(\mathbf{x},
\mathbf{y})\\ \vdots \\ u_{n-1}(\mathbf{x},\mathbf{y})\end{array}\right)
:=D^{-1}\left(\begin{array}{c}y_1\\ y_2\\ \vdots \\ 
y_n\end{array}\right). \label{defu}
\end{align}
Cramer's rule shows us 
\begin{align*}
u_i(\mathbf{x},\mathbf{y})=\Dels^{-1}\cdot\mathrm{det}
\left(\begin{array}{cccccccc}
1 & x_1 & \cdots & x_1^{i-1} & y_1 & x_1^{i+1} & \cdots & x_1^{n-1}\\ 
1 & x_2 & \cdots & x_2^{i-1} & y_2 & x_2^{i+1} & \cdots & x_2^{n-1}\\
\vdots & \vdots & & \vdots & \vdots & \vdots & & \vdots\\
1 & x_n & \cdots & x_n^{i-1} & y_n & x_n^{i+1} & \cdots & x_n^{n-1}
\end{array}\right).
\end{align*}
In order to simplify the presentation, we frequently write 
\[
u_i:=u_i(\mathbf{x},\mathbf{y}),\quad (i=0,\ldots,n-1). 
\]
The Galois group $\Gst$ acts on the orbit $\{u_i^{(\sigma,\tau)}\ |\ 
(\sigma,\tau)\in \Gst \}$ via regular representation from the right. 
However this action is not faithful. 
We put 
\[
H:=\{(\sigma,\tau)\in \Gst\ |\ \varphi(\sigma)=\psi(\tau)\}\cong \frak{S}_n. 
\]
Let $\overline{g}=Hg$ be a right coset of $H$ in $\Gst$. 
If $(\sigma,\tau)\in H$ then we have $u_i^{(\sigma,\tau)}=u_i$ 
for $i=0,\ldots,n-1$ by the lemma below. 
Hence the group $\Gst$ acts on the set 
$\{ u_i^g\ |\ \overline{g}\in H\backslash \Gst\}$ transitively from the right 
through the action on the set $H\backslash \Gst$ of right cosets. 

We see that the set $\{ \overline{(1,\tau)}\ |\ (1,\tau)\in \Gst\}$ 
(resp. $\{ \overline{(\sigma,1)}\ |\ (\sigma,1)\in \Gst\}$) form a complete 
residue system of $H\backslash \Gst$. 
Indeed, for $g=(\sigma,\tau)\in \Gst$, 
$\overline{g}=H(\sigma,\tau')^{-1}g=H(1,(\tau')^{-1}\tau)$ if 
$\varphi(\sigma)=\psi(\tau')$. 

\begin{lemma}\label{stabil}
For $(\sigma,\tau)\in \Gst$ and $i$, $0\leq i\leq n-1$, we have 
$u_i^{(\sigma,\tau)}=u_i$ if and only if $\varphi(\sigma)=\psi(\tau)$; 
that is, $H=\mathrm{Stab}_{\Gst}(u_i)$, the stabilizer of $u_i$ in $\Gst$. 
\end{lemma}
\begin{proof}
Let $A_{i,j}$ be the $(i,j)$-cofactor of the matrix $D$. 
Thus we have the cofactor expansion, 
\begin{align*}
u_i(\mathbf{x},\mathbf{y})=\Dels^{-1}\,\sum_{j=1}^n A_{j,i+1}\ y_j,\quad 
(i=0,\ldots,n-1). 
\end{align*}
Hence for each $i, 0 \leq i \leq n-1$, we see 
\begin{align*}
u_i(\mathbf{x},\mathbf{y})^{(\sigma,1)}
\,=\,\varepsilon(\sigma)\, \Dels^{-1}\,\sum_{j=1}^n A_{j,i+1}^{(\sigma,1)}\ y_j
\,=\,\varepsilon(\sigma)\, \Dels^{-1}\,\sum_{j=1}^n \varepsilon(\sigma)\, A_{\varphi(\sigma)(j),i+1}\ y_j
\end{align*}
where $\varepsilon(\sigma)$ is the signature of $\varphi(\sigma) \in \frak{S}_n$, and also 
\begin{align*}
u_i(\mathbf{x},\mathbf{y})^{(1,\tau^{-1})}
\,&=\,\Dels^{-1}\,\sum_{j=1}^n A_{j,i+1}\ y_j^{\tau^{-1}}\ 
\,=\,\Dels^{-1}\,\sum_{j=1}^n A_{j,i+1}\ y_{\psi(\tau^{-1})(j)}\\
\,&=\,\Dels^{-1}\,\sum_{j=1}^n A_{{\psi(\tau)(j),i+1}}\ y_j. 
\end{align*}
Therefore we have $u_i(\mathbf{x},\mathbf{y})^{(\sigma,\tau)}=u_i(\mathbf{x},\mathbf{y})$ 
if and only if $A_{\varphi(\sigma)(j),i+1}=A_{\psi(\tau)(j),i+1}$ for $j=1,\ldots,n$ 
because the variables $y_1,\ldots,y_n$ are linearly independent over $L_{\mathbf s}$. 
Since $A_{j,i+1}\in k[x_1,\ldots,x_{j-1},x_{j+1},\ldots,x_n]$, we finally have 
$u_i(\mathbf{x},\mathbf{y})^{(\sigma,\tau)}=u_i(\mathbf{x},\mathbf{y})$ 
if and only if $\sigma=\tau$. 
\end{proof}

Hence, in particular, we see that $\#H\backslash G_{\bs,\bt}=n!$ and 
that the subgroups $\Gs$ and $\Gt$ of $\Gst$ act on the set 
$\{ u_i^g\ |\ \overline{g}\in H\backslash \Gst\}$ transitively. 

For $\overline{g}=\overline{(1,\tau)}$, 
we obtain the following equality from the definition (\ref{defu}): 
\[
y_{\psi(\tau)(i)} = u_0^g+u_1^g x_i+\cdots
+u_{n-1}^gx_i^{n-1}\ \ \mathrm{for}\ \ i=1,\ldots,n. 
\]
This means that the set 
$\{(u_0^g,\ldots,u_{n-1}^g)\ |\ \overline{g}\in H\backslash \Gst\}$ 
gives coefficients of $n!$ different Tschirnhausen transformations from 
$f_n(\bs;X)$ to $f_n(\bt;X)$ each of which is respectively defined over 
$K(u_0^g,\ldots,u_{n-1}^g)$. 
\begin{definition}
For each $g\in \Gst$, we call $K(u_0^g,\ldots,u_{n-1}^g)$ a field of 
Tschirnhausen coefficients from $f_n(\bs;X)$ to $f_n(\bt;X)$. 
\end{definition}
We put 
\[
v_i(\mathbf{x},\mathbf{y}):=u_i(\mathbf{y},\mathbf{x}),\ \ \mathrm{for}\ \ i=0,\ldots,n-1. 
\]
We write $v_i=v_i(\mathbf{x},\mathbf{y})$ for simplicity. 
Then $K(v_0^g,\ldots,v_{n-1}^g)$ gives a field of Tschirnhausen 
coefficients from $f_n(\bt;X)$ to $f_n(\bs;X)$. 
We obtain 
\begin{proposition}\label{prop1}
We have $(L_\bs L_\bt)^{g^{-1}Hg}=K(u_0^g,\ldots,u_{n-1}^g)=K(v_0^g,\ldots,v_{n-1}^g)
=K(u_i^g)=K(v_i^g)$ and $[K(u_i^g) : K]=n!$ for each $i, 0\leq i\leq n-1$, and for each 
$g\in \Gst$. 
\end{proposition}
\begin{proof}
We have 
\begin{align*}
&L_\bs L_\bt& &\supset& &
(L_\bs L_\bt)^{g^{-1}Hg}& &\supset& &
K(u_0^g,\ldots,u_{n-1}^g)& &\supset& &K(u_i^g)& &\supset& &K,\\
&\ \ || & & & &\ \ \ || & & & & & & & & & & & &\ \! ||\\
&L_\bs L_\bt& &\supset& &
(L_\bs L_\bt)^{g^{-1}Hg}& &\supset& &
K(v_0^g,\ldots,v_{n-1}^g)& &\supset& &K(v_i^g)& &\supset& &K. 
\end{align*}
Hence the assertion follows from $\mathrm{Stab}_{\Gst}(u_i^g)
=\mathrm{Stab}_{\Gst}(v_i^g)=g^{-1}Hg$. 
\end{proof}
\begin{corollary}
We have $\mathrm{Spl}_{K(u_i^g)} f_n(\bs;X)=\mathrm{Spl}_{K(u_i^g)} f_n(\bt;X)$ 
for every $g\in \Gst$.  
\end{corollary}
\begin{proof}
The polynomials $f_n(\bs;X)$ and $f_n(\bt;X)$ are Tschirnhausen equivalent over 
$K(u_0^g,\ldots,u_{n-1}^g)=K(u_i^g)=K(v_i^g)$. 
Hence the quotient fields $K(u_i^g)[X]/(f_n(\bs;X))$ and $K(u_i^g)[X]/(f_n(\bs;X))$ 
are isomorphic over $K(u_i^g)$. 
\end{proof}
\begin{proposition}
We have 
\begin{align*}
&{\rm (i)}\ L_\bs\cap K(u_i^g)=L_\bt\cap K(u_i^g)=K\ \ \textrm{for}\ \ g\in \Gst;\\
&{\rm (ii)}\ L_\bs L_\bt=L_\bs(u_i^g)=L_\bt(u_i^g)\ \ 
\textrm{for}\ \ g\in \Gst. 
\end{align*}
\end{proposition}
\begin{proof}
(i) We should show that $(g^{-1}Hg)\Gt=(g^{-1}Hg)\Gs=\Gst$. 
We may suppose $g=(1,\tau)$ without loss of generality. 
Then, for any $(\sigma',\tau')\in H$, there exists an element 
$(\sigma',\tau^{-1}\tau'\tau)({\sigma'}^{-1},1)=(1,\tau^{-1}\tau'\tau)$ in 
$(g^{-1}Hg)\Gs$. 
Hence the equality $(g^{-1}Hg)\Gs=\Gst$ follows from $\{ 1\}\times \Gt\subset (g^{-1}Hg)\Gs$. 
The assertion $(g^{-1}Hg)\Gt=\Gst$ is obtained by a similar way because 
we may replace $g$ by an element of the form $(\sigma,1)$.

(ii) We check that $g^{-1}Hg\cap \Gs=g^{-1}Hg\cap \Gt=\{1\}$. 
An element $c$ in $g^{-1}Hg$ can be described as 
$c=(\sigma',\tau^{-1}\tau'\tau)$ where $\varphi(\sigma')=\psi(\tau')$. 
If $c\in \Gs$, then we obtain 
$c=(\sigma',\tau^{-1}\tau'\tau)=(\sigma',1)\in \Gs\times \{ 1\}$. 
Then $\varphi(\sigma')=\psi(\tau')=1$, and hence $c=(1,1)$ in $\Gst$. 
If $c\in \Gt$, then we have $\sigma'=1$. 
Hence follows $c=(1,1)\in \Gst$. 
\end{proof}

Moreover we obtain

\begin{proposition}\label{propLL}
$L_\bs L_\bt=K(u_i^g\ |\ \overline{g}\in H\backslash \Gst)$ for every 
$i, 0\leq i\leq n-1$. 
\end{proposition}
\begin{proof}
We should show that 
$\bigcap_{\overline{g}\in H\backslash \Gst} g^{-1}Hg=\{ 1\}$ 
because $\mathrm{Stab}_{\Gst}(u_i^g)=g^{-1}Hg$. Suppose 
$(\sigma',\tau')\in \bigcap_{\overline{g}\in H\backslash \Gst} g^{-1}Hg$. 
From $H\backslash \Gst=\{\overline{(1,\tau)}\ |\ (1,\tau)\in \Gst\}$, 
we have $\varphi(\sigma')=\psi(\tau^{-1} \tau' \tau)$ 
for every $\tau\in \Gt$. 
Since $\varphi(\sigma')=\psi(\tau)\psi(\tau')\psi(\tau)^{-1}=\psi(\tau)
\varphi(\sigma')\psi(\tau)^{-1}$, 
we see that $\varphi(\sigma')$ is in the center of the symmetric group 
$\frak{S}_n$. 
Thus we have $(\sigma',\tau')=(1,1)$ because the center of $\frak{S}_n$ is trivial. 
\end{proof}
We define a polynomial of degree $n!$ by 
\begin{align*}
F_i(\bs,\bt;X):=\prod_{\overline{g}\in H\backslash \Gst} (X-u_i^g)\in K[X], 
\quad (i=0,\ldots,n-1). 
\end{align*}
We see from Proposition \ref{prop1} that $F_i(\bs,\bt;X), (i=0,\ldots,n-1)$, 
is irreducible over $k(\bs,\bt)$. 
From Proposition \ref{propLL}, furthermore, we have the following theorem: 
\begin{theorem}\label{th-gen}
The polynomial $F_i(\bs,\bt;X) \in k(\bs,\bt)[X]$ is $k$-generic for 
$\frak{S}_n\times \frak{S}_n$. 
\end{theorem}
\begin{proof}
The assertion follows from 
$\mathrm{Spl}_K\, F_i(\bs,\bt;X)=K(u_i^g\ |\ \overline{g}\in H\backslash \Gst)
=L_\bs L_\bt$ and the $\frak{S}_n$-genericness of $f_n(\bs;X)$ and $f_n(\bt;X)$. 
\end{proof}

In case of char $k=2$, we have $k(\bs)(\Dels)=k(\bs)$ because 
$\Dels\in k(\bs)$. 
Hence, we use the results of Berlekamp \cite{Ber76} and take 
the Berlekamp discriminant 
\[
\betas:=\sum_{i<j}\frac{x_i}{x_i+x_j}, 
\]
instead of $\Dels$. 
The field $k(\bs)(\betas)$ is a quadratic extension of $k(\bs)$. 

Let $\mathfrak{A}_n$ be the alternating group of degree $n$, and put 
\begin{align*}
(H\backslash \Gst)^+\, &:=\, \{\overline{g}=\overline{(1,\tau)} 
\in H\backslash \Gst\ |\ \psi(\tau)\in \mathfrak{A}_n\},\\
(H\backslash \Gst)^-\, &:=\, \{\overline{g}=\overline{(1,\tau)} 
\in H\backslash \Gst\ |\ \psi(\tau)\not\in \mathfrak{A}_n\} 
\end{align*}
and 
\begin{align}
F_i^{\pm}(X)\, :=\prod_{\overline{g}\in (H\backslash \Gst)^\pm} 
(X-u_i^g), \quad (i=0,\ldots,n-1).\label{defFpm}
\end{align}
\begin{proposition}\label{prop-decom}
If char $k\neq 2$ $($resp. char $k=2$$)$, the polynomial 
$F_i(\bs,\bt;X)$ splits into two irreducible factors $F_i^+(X)$ and 
$F_i^-(X)$ of degree $n!/2$ over the quadratic extension $K(\Dels/\Delt)$ 
$($resp. $K(\betas+\betat))$ of $K$. 
\end{proposition}
\begin{proof}
Let $\Gs^+$ (resp. $\Gt^+$) be the subgroup of $\Gs$ 
(resp. $\Gt$) with index two which is isomorphic to the alternating 
group $\frak{A}_n$ of degree $n$. 
If char $k\neq 2$, we have a sequence of subgroups $\Gs^+\times 
\Gt^+\subset \mathrm{Stab}_{\Gs\times \Gt}(\Dels/\Delt)
\subset \Gst$ each of whose indices is two. 
It follows from $H\subset \mathrm{Stab}_{\Gst}(\Dels/\Delt)$ that $F_i^+(X)\ |\ F_i(X)$ 
and $F_i^+(X)\in K(\Dels/\Delt)$. 
The case of char $k=2$ may easily be obtained by a similar manner. 
\end{proof}

\section{Specialization of parameters into an infinite field}\label{seSpe}

Let $M\ (\supset k)$ be an infinite field. 
We assume that we always take a specialization of parameters 
$(\bs,\bt) \mapsto (\ba,\bb)\in M^n\times M^n$ so that both of 
the polynomials $f_n(\ba;X)$ and $f_n(\bb;X)$ are separable 
over $M$, (i.e. $\Delta_\ba\neq 0$ and $\Delta_\bb\neq 0$). 
Put $L_{\ba}=\mathrm{Spl}_M\, f_n(\ba;X)$ and $L_{\bb}=\mathrm{Spl}_M\, f_n(\bb;X)$; 
these splitting fields are taken in a fixed algebraic closure of $M$. 
We denote the Galois groups of $f_n(\ba;X)$ and $f_n(\bb;X)$ over $M$ 
by $G_\ba$ and $G_\bb$ respectively, that is, $G_\ba=\mathrm{Gal}(L_{\ba}/M)$ and 
$G_\bb=\mathrm{Gal}(L_{\bb}/M)$. 
We put $G_{\ba,\bb}:=\mathrm{Gal}(L_{\ba}L_{\bb}/M)$. 
Let $\bm{\alpha}:=(\alpha_1,\ldots,\alpha_n)$ (resp. $\bm{\beta}:=
(\beta_1,\ldots,\beta_n)$) be the roots of $f_n(\ba;X)$ (resp. $f_n(\bb;X)$) 
in the algebraic closure of $M$. 
Once we fix the orders of the roots as $\bm{\alpha}$ and $\bm{\beta}$, then each element 
of $G_{\ba,\bb}$ induces substitutions on the two sets of indices.  Then we may identify 
$G_{\ba,\bb}$ as a subgroup of $\Gst$. 
More precisely, we express each element $h \in G_{\ba,\bb}$ as 
$h=(\sigma, \tau) \in \Gst$ through the conditions, 
$\alpha_i^h=\alpha_{\varphi(\sigma)(i)}$ and 
$\beta_i^h=\beta_{\psi(\tau)(i)}$ for $i=1, \ldots, n$. 
We put for $g = (\sigma, \tau) \in \Gst$ 
\begin{align}
(c_0^g,\ldots,c_{n-1}^g)\,&:=\,(u_0^g(\bm{\alpha},\bm{\beta}),\ldots,
u_{n-1}^g(\bm{\alpha},\bm{\beta})),\label{utoc}\\
(d_0^g,\ldots,d_{n-1}^g)\,&:=\,(u_0^g(\bm{\beta},\bm{\alpha}),\ldots,
u_{n-1}^g(\bm{\beta},\bm{\alpha})). \nonumber
\end{align}
Then we have 
\begin{align}
\beta_{\psi(\tau)(i)}\,&=\, c_0^g + c_1^g\alpha_{\varphi(\sigma)(i)} + \cdots + 
c_{n-1}^g\alpha_{\varphi(\sigma)(i)}^{n-1}, \label{b=a's}\\
\alpha_{\varphi(\sigma)(i)}\,&=\, d_0^g + d_1^g\beta_{\psi(\tau)(i)} + \cdots + 
d_{n-1}^g\beta_{\psi(\tau)(i)}^{n-1} \label{a=b's}
\end{align}
for each $i = 1, \ldots, n$. 
For each $g\in G_{\bs,\bt}$, there exists a Tschirnhausen transformation 
from $f_n(\ba;X)$ to $f_n(\bb;X)$ over its field of Tschirnhausen 
coefficients $M(c_0^g,\ldots,c_{n-1}^g)$, and the $n$-tuple 
$(d_0^g,\ldots,d_{n-1}^g)$ gives the coefficients of a transformation 
of the inverse direction. 

From the assumption $\Delta_\ba\cdot \Delta_\bb\neq 0$, we first see 
\begin{lemma}\label{lemM}
Let $M'/M$ be a field extension. 
If $f_n(\bb;X)$ is a Tschirnhausen transformation of $f_n(\ba;X)$ over 
$M'$, then $f_n(\ba;X)$ is a Tschirnhausen transformation of $f_n(\bb;X)$ 
over $M'$. 
In particular, we have $M(c_0^g,\ldots,c_{n-1}^g)=M(d_0^g,\ldots,d_{n-1}^g)$ 
for every $g=(\sigma, \tau) \in G_{\bs,\bt}$. 
\end{lemma}
\begin{proof}
Suppose $M'$ contains $M(c_0^g,\ldots,c_{n-1}^g)$. 
We will take a ring homomorphism 
\[
\rho\,:\,M'[X]/(f_n(\bb;X))\longrightarrow M'[X]/(f_n(\ba;X)),\ X\longmapsto 
c_0^g + c_1^gX + \cdots + c_{n-1}^gX^{n-1}
\]
and show that the map $\rho$ is an isomorphism over $M'$. 

We first assume that $f_n(\ba;X)$ (resp. $f_n(\bb;X))$ splits into irreducible factors 
$g_{\mu}(X)$ (resp. $h_{\nu}(X)$) for $1\leq \mu\leq m$ 
(resp. $1\leq \nu\leq m'$) over $M'$. 
Since $f_n(\ba;X)$ and $f_n(\bb;X)$ have no multiple root, the quotient algebras 
$M'[X]/(f_n(\ba;X))$ and $M'[X]/(f_n(\bb;X))$ are semi-simple and direct sums of the fields 
$M'[X]/(g_{\mu}(X))$ and $M'[X]/(h_{\nu}(X))$, respectively. 
We put $L' = \mathrm{Spl}_{M'} f_n(\ba;X)\, \mathrm{Spl}_{M'} f_n(\bb;X)$. 
Fix $i, 0\leq i \leq n-1$, and suppose that $\alpha_{\varphi(\sigma)(i)}$ and 
$\beta_{\psi(\tau)(i)}$ are roots of the irreducible factors $g_{\mu}(X)$ 
and $h_{\nu}(X)$, respectively. Then we have embeddings $\Phi_i : M'[X]/(g_{\mu}(X))\to L'$ 
with $\Phi_i(X)=\alpha_{\varphi(\sigma)(i)}$ and $\Psi_i : M'[X]/(h_{\nu}(X)) \to L'$ with 
$\Psi_i(X)=\beta_{\psi(\tau)(i)}$. 
By the equality (\ref{b=a's}) we see 
\[
\Psi_i(X\ \mathrm{mod} (h_{\nu}(X)))=\Phi_i(c_0^g + c_1^gX + \cdots + c_{n-1}^gX^{n-1}\ 
\mathrm{mod} (g_{\mu}(X))).
\]

Thus we get an injective homomorphism from $M'[X]/(h_{\nu}(X))$ to $M'[X]/(g_{\mu}(X))$ 
by assigning $c_0^g + c_1^gX + \cdots + c_{n-1}^gX^{n-1}\ \mathrm{mod} (g_{\mu}(X))$ to 
$X\ \mathrm{mod} (h_{\nu}(X))$. 
Since the quotient algebras $M'[X]/(f_n(\ba;X))$ and $M'[X]/(f_n(\bb;X))$ are direct sums 
of the fields $M'[X]/(g_{\mu}(X))$ and $M'[X]/(h_{\nu}(X))$, respectively, the homomorphism 
$\rho$ as above is well defined and injective. 
Then this has to be an isomorphism over $M'$ because the dimensions of the two algebras 
over $M'$ coincide. Therefore, in particular, the corresponding fields $M'[X]/(g_{\mu}(X))$ 
and $M'[X]/(h_{\nu}(X))$ are isomorphic over $M'$. 
In particular, each irreducible factor $g_{\mu}(X)$ of $f_n(\ba;X)$ corresponds to an 
irreducible factor $h_{\nu}(X)$ of $f_n(\bb;X)$ in a one-to-one manner. 
The degrees of the two factors are equal.  
Hence we also see that the number of irreducible factors of $f_n(\ba;X)$ is same as that of 
$f_n(\bb;X)$, that is, $m=m'$. 

We take the inverse isomorphism from $M'[X]/(f_n(\ba;X))$ 
to $M'[X]/(f_n(\bb;X))$ over $M'$.  Then the image of $X\ \mathrm{mod} (f_n(\ba;X))$ is 
expressed as $e_0 + e_1X + \cdots + e_{n-1}X^{n-1}\ \mathrm{mod} (f_n(\bb;X))$ with 
$e_0, \ldots, e_{n-1} \in M'$. 
The homomorphisms $\Phi_i$ and $\Psi_i$ now give us 
\[
\alpha_{\varphi(\sigma)(i)} = e_0 + e_1\beta_{\psi(\tau)(i)} + \cdots + 
e_{n-1}\beta_{\psi(\tau)(i)}^{n-1} 
\]
for $i = 1, \ldots, n$. Since $\Delta_{\bb} \neq 0$, therefore, these equalities 
together with $(\ref{a=b's})$ show $d_j^g = e_j \in M'$ for $j = 0, \ldots, n-1$. 

In particular if we take $M'=M(c_0^g,\ldots,c_{n-1}^g)$ then we see 
$M(c_0^g,\ldots,c_{n-1}^g)$ $\supset$ $M(d_0^g,\ldots,d_{n-1}^g)$. 
Conversely if we take $M'=M(d_0^g,\ldots,d_{n-1}^g)$ then we have 
$M(d_0^g,\ldots,d_{n-1}^g)$ $\supset$ $M(c_0^g,\ldots,c_{n-1}^g)$. 
\end{proof}

Our main idea of this paper is to study the behavior of the field 
$M(c_0^g,\ldots,c_{n-1}^g)$ of Tschirnhausen coefficients 
from $f_n(\ba;X)$ to $f_n(\bb;X)$. 
\begin{proposition}\label{propc}
Under the assumption, $\Delta_\ba\cdot\Delta_\bb\neq 0$, we have 
the following two assertions\,{\rm :} 
\begin{align*}
&{\rm (i)}\ \ \mathrm{Spl}_{M(c_0^g,\ldots,c_{n-1}^g)} f_n(\ba;X)
=\mathrm{Spl}_{M(c_0^g,\ldots,c_{n-1}^g)} f_n(\bb;X)\ \ \textit{for each}\ 
\ g\in \Gst\, {\rm ;}\\
&{\rm (ii)}\ L_\ba L_\bb=L_\ba\, M(c_0^g,\ldots,c_{n-1}^g)=L_\bb\, 
M(c_0^g,\ldots,c_{n-1}^g)\ \ \textit{for each}\ \ g\in \Gst.
\end{align*}
\end{proposition}
\begin{proof}
Put $M'=M(c_0^g,\ldots,c_{n-1}^g)$. 
Then, by Lemma \ref{lemM}, $M'[X]/(f_n(\ba;X))$ and $M'[X]/(f_n(\bb;X))$ 
are isomorphic over $M'$. Hence follows the assertion (i). 
By (i) we see that 
$L_\ba\, M(c_0^g,\ldots,c_{n-1}^g)=L_\bb\, M(c_0^g,\ldots,c_{n-1}^g)$. 
Thus the assertion (ii) also holds. 
\end{proof}

It follows from Proposition \ref{prop1} and Proposition \ref{propLL} that, 
for a fixed $j, (0\leq j\leq n-1)$, we have 
\begin{align}
K(u_0^g,\ldots,u_{n-1}^g)\, &=\, K(u_j^g)\quad \mathrm{for}\quad g\in \Gst \label{equ},\\
L_\bs L_\bt\, &=\, K(u_j^g\ |\ \overline{g}\in H\backslash \Gst).\nonumber
\end{align}
and $[K(u_j^g) : K]=n!$. 
After the specialization as in (\ref{utoc}) we may only have inclusion relations 
\begin{align*}
M(c_0^g,\ldots,c_{n-1}^g)\, &\supset\, M(c_j^g)\quad \mathrm{for}\quad g\in \Gst,\\
L_\ba L_\bb\, &\supset \, M(c_j^g\ |\ \overline{g}\in H\backslash \Gst). 
\end{align*}
Whether the equality $M(c_0^g,\ldots,c_{n-1}^g)=M(c_j^g)$ holds or not depends on 
the specialization $(\bs,\bt)\mapsto (\ba,\bb)\in M^n\times M^n$. 
By (\ref{equ}) there exists $P_{i,j}(\bs,\bt;X)\in K[X]$ such that 
\begin{align*}
u_i=P_{i,j}(\bs,\bt;u_j)\quad \mathrm{and}\quad \mathrm{deg}_X(P_{i,j}(\bs,\bt;X))<n!. 
\end{align*}
Then we take polynomials $P_{i,j}^0(\bs,\bt;X)\in k[\bs,\bt][X]$ and 
$D_{i,j}^0(\bs,\bt)\in k[\bs,\bt]$ which satisfy 
\begin{align}
u_i=\frac{1}{D_{i,j}^0(\bs,\bt)}\, P_{i,j}^0(\bs,\bt;u_j),\quad \mathrm{and}\quad 
\mathrm{deg}_X(P_{i,j}^0(\bs,\bt;X))<n! \label{eqD0}
\end{align}
by extracting the minimal multiple of the denominators of the coefficients 
of the polynomial $P_{i,j}(\bs,\bt;X)$ in $X$. 
Then the following lemma is clear. 
\begin{lemma}\label{lemD0}
For a fixed $j, 0\leq j \leq n-1$, and $\ba,\bb\in M^n$, the condition
$D_{i,j}^0(\ba,\bb)\neq 0$ for $i=0,\ldots,n-1$ implies 
$M(c_0^g,\ldots,c_{n-1}^g)=M(c_j^g)$ for each $g\in \Gst$. 
\end{lemma}

After the specialization as in (\ref{utoc}), we also utilize the polynomial
\begin{align*}
F_i(\ba,\bb;X)\, =\, \prod_{\overline{g}\in H\backslash \Gst} (X-c_i^g)\in M[X], \quad 
(i=0,\ldots,n-1)
\end{align*}
of degree $n!$ which is not necessary irreducible. 

\begin{lemma}\label{lemnoMult}
If $F_j(\ba,\bb;X)$ has no multiple root for a fixed $j, 0\leq j \leq n-1$ 
and $\ba,\bb\in M^n$ with $\Delta_\ba\cdot\Delta_\bb\neq 0$ then 
$D_{i,j}^0(\ba,\bb)\neq 0$ for $i=0,\ldots,n-1$. 
\end{lemma}
\begin{proof}
By (\ref{eqD0}), if $D_{i,j}^0(\ba,\bb)=0$ for an $i$, then we should have 
$P_{i,j}^0(\ba,\bb;c_j^g)=0$ for $g\in\Gst$. 
From the assumption $c_j^g\neq c_j^h$ $(\overline{g}\neq\overline{h})$, 
$n!$ different $c_j^g$ satisfy the equality $P_{i,j}^0(\ba,\bb;c_j^g)=0$ of 
degree less than $n!$. This is a contradiction. 
\end{proof}

\begin{proposition}
For a fixed $j, 0\leq j \leq n-1$, and $\ba,\bb \in M^n$ with 
$\Delta_\ba\cdot\Delta_\bb\neq 0$, suppose $D_{i,j}^0(\ba,\bb)\neq 0$ 
for $i= 0,\ldots, n-1$.  Then $F_j(\ba,\bb;X)$ has no multiple root. 
\end{proposition}
\begin{proof}
We first note that $\{c_j^g\ |\ \overline{g}\in H\backslash \Gst\}
=\{c_j^{(1,\tau)}\ |\ \tau\in \Gt\}$. 
It follows from the condition, $\Delta_\ba\cdot\Delta_\bb\neq 0$, 
the roots $\beta_1,\ldots,\beta_n$ of $f_n(\bb;X)$ are distinct. 
For $\tau, \tau' \in \Gt$, therefore, $(\beta_{\psi(\tau)(1)},\ldots,\beta_{\psi(\tau)(n)})
=(\beta_{\psi(\tau')(1)},\ldots,\beta_{\psi(\tau')(n)})$ if and only if $\tau=\tau'$. 
Since $\mathbf{y}^{\tau}=(y_{\psi(\tau)(1)},\ldots,y_{\psi(\tau)(n)})$, 
we have, by the definition, 
\begin{align*}
\left(\begin{array}{c}
c_0^{(1,\tau)}\\ 
c_1^{(1,\tau)}\\ 
\vdots \\ 
c_{n-1}^{(1,\tau)}
\end{array}\right)
\,=\,\left(
\begin{array}{ccccc}
1 & \alpha_1 & \alpha_1^2 & \cdots & \alpha_1^{n-1}\\ 
1 &\alpha_2 & \alpha_2^2 & \cdots & \alpha_2^{n-1}\\
\vdots & \vdots & \vdots & \ddots & \vdots\\
1 & \alpha_n & \alpha_n^2 & \cdots & \alpha_n^{n-1}
\end{array}\right)^{-1}
\left(\begin{array}{c}\beta_{\psi(\tau)(1)}\\ 
\beta_{\psi(\tau)(2)}\\ 
\vdots \\ 
\beta_{\psi(\tau)(n)}
\end{array}\right). 
\end{align*}
We also have the equation for $\tau'$ if we replace $\tau$ by $\tau'$. 
On the other hand, we have 
\begin{align*}
c_i^{(1,\tau)}\,&=\,\frac{1}{D_{i,j}^0(\ba,\bb)}\, P_{i,j}^0(\ba,\bb;c_j^{(1,\tau)}),\\
c_i^{(1,\tau')}\,&=\,\frac{1}{D_{i,j}^0(\ba,\bb)}\, P_{i,j}^0(\ba,\bb;c_j^{(1,\tau')}),
\end{align*}
from the condition $D_{i,j}^0(\ba,\bb)\neq 0$ for $i=0,\ldots,n-1$. 
If $c_j^{(1,\tau)}=c_j^{(1,\tau')}$, therefore, we have 
$(c_0^{(1,\tau)},\ldots,c_{n-1}^{(1,\tau)})
=(c_0^{(1,\tau')},\ldots,c_{n-1}^{(1,\tau')})$. 
Hence from the above equalities, we have $(\beta_{\psi(\tau)(1)},\ldots,\beta_{\psi(\tau)(n)})
=(\beta_{\psi(\tau')(1)},\ldots,\beta_{\psi(\tau')(n)})$. 
Thus we see $\tau=\tau'$ and the assertion of the proposition. 
\end{proof}
Before we go on farther analysis, we explain the actions of $\Gst$ and $\Gab$ on the set 
of values $\{c_i^g\ |\ \overline{g}\in H\backslash \Gst\}, 0\leq i \leq n-1$, where 
$g \in \Gst$ runs over a set of representatives of the cosets $H\backslash \Gst$. 
We defined $c_i^g$ by $c_i^g = u_i^g(\bm{\alpha}, \bm{\beta})$. 
Since the rational function $u_i(\mathbf{x}, \mathbf{y})$ of $2n$ variables 
$\mathbf{x}=(x_1,\ldots, x_n),\, \mathbf{y}=(y_1,\ldots, y_n)$ belongs to 
$L_\bs L_\bt=k(\mathbf{x}, \mathbf{y})$, the Galois group 
$\Gst = \mathrm{Gal}(L_\bs L_\bt/k(\mathbf{s}, \mathbf{t}))$ naturally acts on 
$u_i(\mathbf{x}, \mathbf{y})$; since $g$ induces substitutions on the sets 
$\{x_1,\ldots, x_n\}$ and $\{y_1,\ldots, y_n\}$, we express $g=(\sigma, \tau)$ and 
$u_i^g(\mathbf{x}, \mathbf{y})=u_i(\mathbf{x}^g, \mathbf{y}^g)$ with 
$\mathbf{x}^g=(x_{\varphi(\sigma)(1)},\ldots, x_{\varphi(\sigma)(n)})$ and 
$\mathbf{y}^g=(y_{\psi(\tau)(1)},\ldots, x_{\psi(\tau)(n)})$. 
However, $\Gst$ does not directly act on the set of values $\{c_i^g\ |\ 
\overline{g}\in H\backslash \Gst\}$ of $u_i^g$.  
We consider the collection of the values $c_i^g$, $\overline{g}\in H\backslash \Gst$, as 
a function $\mathbf{c}_i(g):=c_i^g$ defined on the cosets $H\backslash \Gst$ because $u_i^g$ 
and hence $c_i^g$ depend only on the coset $Hg$. Then for $h \in \Gst$, 
$\mathbf{c}_i^h(g):=c_i^{gh}$ is the composition of $\mathbf{c}_i$ with the translation on 
$H\backslash \Gst$ by $h$. Therefore each $h \in \Gst$ induces a substitution of the set 
of values $\{c_i^g\ |\ \overline{g}\in H\backslash \Gst\}$. This is the way by which 
$\Gst$ acts on the set. 
Hence, in particular, its action is transitive. 
Thus if $\#\{c_i^g\ |\ \overline{g}\in H\backslash \Gst\}=n!$ then 
we have $\mathrm{Stab}_{\Gst}(c_i^g)=\mathrm{Stab}_{\Gst}(u_i^g(\mathbf{x}, \mathbf{y}))
=g^{-1}Hg$ because $\#(H\backslash \Gst)=n!$. 

As for $\Gab$, the situation is different. It acts on the set of values 
$\{c_i^g\ |\ \overline{g}\in H\backslash \Gst\}$ as the Galois group 
$\mathrm{Gal}(L_\ba L_\bb/M)$ because every $c_i^g$ is contained in the field $L_\ba L_\bb$. 
Under the assumption, $\Delta_\ba\cdot\Delta_\bb\neq 0$, furthermore, we regarded $\Gab$ 
as a subgroup of $\Gst$. 
\begin{lemma}
The action of $\Gab$ on the set
$\{c_i^g\ |\ \overline{g}\in H\backslash \Gst\}$ as the Galois group 
$\mathrm{Gal}(L_\ba\,L_\bb/M)$ coincides with the one as a subgroup of $\Gst$ through 
the action on the set of cosets $H\backslash \Gst$. 
\end{lemma}
\begin{proof}
For $g=(\sigma, \tau) \in \Gst$, we have by the definition 
\begin{align*}
\left(\begin{array}{c}c_0^g\\ c_1^g\\ \vdots \\ c_{n-1}^g \end{array}\right)
=\left(
\begin{array}{ccccc}
1 & \alpha_{\varphi(\sigma)(1)} & \alpha_{\varphi(\sigma)(1)}^2 & \cdots & 
\alpha_{\varphi(\sigma)(1)}^{n-1}\\ 
1 & \alpha_{\varphi(\sigma)(2)} & \alpha_{\varphi(\sigma)(2)}^2 & \cdots & 
\alpha_{\varphi(\sigma)(2)}^{n-1}\\
\vdots & \vdots & \vdots & \ddots & \vdots\\
1 & \alpha_{\varphi(\sigma)(n)} & \alpha_{\varphi(\sigma)(n)}^2 & \cdots & 
\alpha_{\varphi(\sigma)(n)}^{n-1}\end{array}\right)^{-1}
\left(\begin{array}{c}\beta_{\psi(\tau)(1)}\\ \beta_{\psi(\tau)(2)}\\ \vdots \\ 
\beta_{\psi(\tau)(n)}\end{array}\right). 
\end{align*}
An element $h \in \Gab$ is identified as an element $h=(\sigma', \tau')\in\Gst$ through 
$(\alpha_i)^h = \alpha_{\varphi(\sigma')(i)}$ and $(\beta_i)^h = \beta_{\psi(\tau')(i)}$ 
for $i=1,\ldots, n$. 
Since $\varphi$ and $\psi$ are anti-isomorphisms, we have 
$\varphi(\sigma')\varphi(\sigma)=\varphi(\sigma\sigma')$ and 
$\psi(\tau')\psi(\tau)=\psi(\tau\tau')$. 
Hence the action of $h$ on $c_i^g, i=0,\ldots, n-1$, as an element of 
$\mathrm{Gal}(L_\ba\,L_\bb/M)$ is given by 
\begin{align*}
\left(\begin{array}{c}(c_0^g)^h\\ (c_1^g)^h\\ \vdots \\ (c_{n-1}^g)^h \end{array}\right)
=\left(
\begin{array}{ccccc}
1 & \alpha_{\varphi(\sigma \sigma')(1)} & \alpha_{\varphi(\sigma \sigma')(1)}^2 
& \cdots & \alpha_{\varphi(\sigma \sigma')(1)}^{n-1}\\ 
1 & \alpha_{\varphi(\sigma \sigma')(2)} & \alpha_{\varphi(\sigma \sigma')(2)}^2 
& \cdots & \alpha_{\varphi(\sigma \sigma')(2)}^{n-1}\\
\vdots & \vdots & \vdots & \ddots & \vdots\\
1 & \alpha_{\varphi(\sigma \sigma')(n)} & \alpha_{\varphi(\sigma \sigma')(n)}^2 
& \cdots & \alpha_{\varphi(\sigma \sigma')(n)}^{n-1}\end{array}\right)^{-1}
\left(\begin{array}{c}\beta_{\psi(\tau \tau')(1)}\\ \beta_{\psi(\tau \tau')(2)}\\ 
\vdots \\ \beta_{\psi(\tau \tau')(n)}\end{array}\right). 
\end{align*}
Therefore we see $(c_i^g)^h=c_i^{gh}, i=0,\ldots, n-1$. 
\end{proof}
\begin{proposition}\label{prop12}
Assume as above that $\Delta_\ba\cdot\Delta_\bb\neq 0$ for $\ba,\bb \in M^n$. 
Suppose that the polynomial $F_j(\ba,\bb;X)$ has no multiple root for some 
$j, 0\leq j \leq n-1$. Then the following two assertions hold\,{\rm :}\\
$(\mathrm{i})$\ $M(c_0^g,\ldots,c_{n-1}^g)=M(c_j^g)$ for each $g \in \Gst$\,{\rm ;}\\
$(\mathrm{ii})$ $L_\ba L_\bb=M(c_j^g\ |\ \overline{g}\in H\backslash \Gst)$. 
\end{proposition}
\begin{proof}
The assertion (i) follows from Lemma \ref{lemD0} and Lemma \ref{lemnoMult}. 
Here we give an alternative proof of (i) and a proof of (ii) from the viewpoint of 
above discussion.

(i) If $F_j(\ba,\bb;X)$ has no multiple root, then 
$\#\{c_j^g\ |\ \overline{g}\in H\backslash \Gst\}=n!$. 
Therefore, we have $\mathrm{Stab}_{\Gab}(c_j^g)=g^{-1}Hg\,\cap\,\Gab$ 
for each $g \in \Gst$. 
It follows from Lemma \ref{stabil} that $g^{-1}Hg\,\cap\,\Gab$ is contained 
in $\mathrm{Stab}_{\Gab}(c_i^g)$ for every $i, 0\leq i\leq n-1$. Therefore we 
see $\mathrm{Stab}_{\Gab}(c_j^g) \subset \mathrm{Stab}_{\Gab}(c_i^g)$ for every 
$i, 0\leq i \leq n-1$, and $M(c_j^g) \supset M(c_i^g)$ for every $i, 0\leq i \leq n-1$. 
This shows (i). 

(ii) By the definition of $c_j^g$, we have 
$L_\ba L_\bb\supset M(c_j^g\ |\ \overline{g}\in H\backslash G_{\bs,\bt})\supset M
={L_\ba L_\bb}^{G_{\ba,\bb}}$. 
Let $h$ be an element of $\Gab$ and suppose that $h$ is trivial on 
$M(c_j^g\ |\ \overline{g}\in H\backslash \Gst)$. 
As we saw above, we have  $\mathrm{Stab}_{\Gab}(c_j^g)=g^{-1}Hg\,\cap\,\Gab$ 
for each $g \in \Gst$. 
Therefore $h$ belongs to $\bigcap_{\overline{g}\in H\backslash \Gst} g^{-1}Hg$ which 
is equal to $\{ 1 \}$ as we showed it in the proof of Proposition \ref{propLL}. 
Hence we have $h=1$ and then $\mathrm{Gal}(L_\ba L_\bb/M(c_j^g\ |\ \overline{g}\in 
H\backslash \Gst))=\{ 1 \}$. 
This shows (ii). 
\end{proof}
\begin{theorem}\label{throotf}
For a fixed $j, 0\leq j \leq n-1$, and $\ba,\bb \in M^n$ with 
$\Delta_\ba\cdot\Delta_\bb\neq 0$, assume that the polynomial $F_j(\ba,\bb;X)$ has no 
multiple root. Then $F_j(\ba,\bb;X)$ has a root in $M$ if and only if $M[X]/(f_n(\ba;X))$ 
and $M[X]/(f_n(\bb;X))$ are $M$-isomorphic. 
\end{theorem}
\begin{proof}
First suppose that one of the roots of $F_j(\ba,\bb;X)$, $c_j^g$ for some $g \in \Gst$, 
is in $M$. By the preceding proposition, we see $c_0^g, \ldots , c_{n-1}^g \in M$. 
Express $g=(\sigma, \tau)$ as usual. 
Then for the root $\alpha_{\varphi(\sigma)(i)}$ of $f_n(\ba;X)$, $\beta_{\psi(\tau)(i)} 
:= c_0^g + c_1^g\alpha_{\varphi(\sigma)(i)} + \cdots + 
c_{n-1}^g\alpha_{\varphi(\sigma)(i)}^{n-1}$ is a root of $f_n(\bb;X)$ in $M(\alpha)$ 
for $i=1,\ldots, n$. It is now easy to adopt the argument in the proof of Lemma \ref{lemM} 
and to obtain an isomorphism from $M[X]/(f_n(\bb;X))$ to $M[X]/(f_n(\ba;X))$ over $M$ by 
assigning $c_0^g + c_1^gX + \cdots + c_{n-1}^gX^{n-1}\ \mathrm{mod} (f_n(\ba;X))$ to 
$X\ \mathrm{mod} (f_n(\bb;X))$. 

Now suppose conversely that $M[X]/(f_n(\ba;X))$ 
and $M[X]/(f_n(\bb;X))$ are $M$-isomorphic. 
Then as in the proof of Lemma \ref{lemM}, the irreducible factors of $f_n(\ba;X)$ and those 
of $f_n(\bb;X)$ perfectly correspond. 
Moreover, each pair of corresponding simple components are isomorphic over $M$. 
Denote the image of $X\ \mathrm{mod} (f_n(\bb;X))$ by the isomorphism 
from $M[X]/(f_n(\bb;X))$ to $M[X]/(f_n(\ba;X))$ by 
$\eta := e_0 + e_1X + \cdots + e_{n-1}X^{n-1}\ \mathrm{mod} (f_n(\ba;X))$ 
with $e_0, \ldots , e_{n-1} \in M$. Then we find $\tau \in \Gt$ so that we 
have $\beta_{\psi(\tau)(i)} = e_0 + e_1\alpha_i + \cdots + e_{n-1}\alpha_i^{n-1}$ for 
$i=1,\ldots, n$. Hence for $g=(1, \tau) \in \Gst$, 
we must have $e_{\nu}=c_{\nu}^g$ for $\nu = 0,\ldots, n-1$. In particular, 
we see that $c_j^g$ is a root of $F_j(\ba,\bb;X)$ in $M$. 
\end{proof}
In the case where $G_\ba$ and $G_\bb$ are isomorphic to a transitive subgroup $G$ of 
$\frak{S}_n$ and every subgroups of $G$ with index $n$ are conjugate in $G$, 
the condition that $M[X]/(f_n(\ba;X))$ and $M[X]/(f_n(\bb;X))$ are isomorphic over $M$ is 
equivalent to the condition that $\mathrm{Spl}_M f_n(\ba;X)$ and 
$\mathrm{Spl}_M f_n(\bb;X)$ coincide. 
Hence we get an answer to the field isomorphism problem via $F_j(\ba,\bb;X)$. 
\begin{corollary}\label{cor1}
Let $j$ and $\ba,\bb \in M^n$ be as in Theorem \ref{throotf}. 
Assume that both of $G_\ba$ and $G_\bb$ are isomorphic to a transitive subgroup $G$ 
of $\frak{S}_n$ and that all subgroups of $G$ with index $n$ are conjugate in $G$. 
Then $F_j(\ba,\bb;X)$ has a root in $M$ if and only if 
$\mathrm{Spl}_M f_n(\ba;X)$ and $\mathrm{Spl}_M f_n(\bb;X)$ 
coincide. 
\end{corollary}
We note that if $G$ is one of the symmetric group 
$\frak{S}_n$, ($n\neq 6$), the alternating group $\frak{A}_n$ of degree 
$n$, ($n\neq 6$), and solvable transitive subgroups of $\mathfrak{S}_p$ 
of prime degree $p$, then all subgroups of $G$ with index $n$ or $p$, respectively, 
are conjugate in $G$ (cf. \cite{Hup67}, \cite{BJY86}). 

Let $H_1$ and $H_2$ be subgroups of $\mathfrak{S}_n$. 
As an analogue to Theorem \ref{th-gen}, we obtain a $k$-generic polynomial 
for $H_1\times H_2$, the direct product of groups $H_1$ and $H_2$. 
\begin{theorem}\label{thgen}
Let $M=k(q_1,\ldots,q_l,r_1,\ldots,r_m)$, $(1\leq l,\, m\leq n-1)$ be the rational function 
field over $k$ with $(l+m)$ variables. 
For $\ba\in {k(q_1,\ldots,q_l)}^n, \bb\in {k(r_1,\ldots,r_m)}^n$, we assume that 
$f_n(\ba;X)\in M[X]$ and $f_n(\bb;X)\in M[X]$ be 
$k$-generic polynomials for $H_1$ and $H_2$ respectively. 
For a fixed $j, 0\leq j \leq n-1$, assume that $F_j(\ba,\bb;X)\in M[X]$ 
has no multiple root. 
Then $F_j(\ba,\bb;X)$ is a $k$-generic polynomial for $H_1\times H_2$ which is not 
necessary irreducible. 
\end{theorem}
\begin{proof}
It follows from Proposition \ref{prop12} that 
$M(c_j^g\ |\ \overline{g}\in H\backslash G_{\bs,\bt})=L_\ba L_\bb$. 
Hence the assertion follows from the $H_1$-genericness of $f_n(\ba;X)$ and 
the $H_2$-genericness of $f_n(\bb;X)$. 
\end{proof}

In each Tschirnhausen equivalence class, we can always choose a 
specialization $\bs\mapsto \ba\in M^n$ of the polynomial $f_n(\bs;X)$ 
which satisfy $a_1=0$ and $a_{n-1}=a_{n}$ (see \cite[\S 8.2]{JLY02}). 
Hence the polynomial 
\begin{align*}
&g_n(q_2,\ldots,q_{n-1};X)\\
&:=(-1)^n\cdot f_n(0,q_2,\ldots,q_{n-2},q_{n-1},q_{n-1};-X)\\
&=X^n+q_2X^{n-2}+\cdots+q_{n-2}X^2+q_{n-1}X+q_{n-1}
\end{align*}
is $k$-generic for $\frak{S}_n$ with $(n-2)$ parameters 
$q_2,\ldots,q_{n-1}$ over an arbitrary field $k$. 

Indeed if the characteristic of the field $k$ is prime to $n$, 
we obtain $q_2,\ldots,q_{n-1}$ in terms of $s_1,\ldots,s_n$ as follows: 
we put $\mathbf{X}:=(X_1,\ldots,X_n)$, 
\[
X_1:=x_1-s_1/n,\ X_2:=x_2-s_1/n,\ \ldots,\ X_n:=x_n-s_1/n: 
\]
then we have $k(\mathbf{X}):=k(X_1,\ldots,X_{n-1})\subset k(\mathbf{x})$ and 
\[
X_1+X_2+\cdots+X_n=0.
\]
The $\mathfrak{S}_n$-action on $k(\mathbf{x})$ induces an action on 
$k(\mathbf{X})$ which is linear and faithful. We also have 
$k(\mathbf{X})^{\frak{S}_n}=k(\mathbf{S}):=k(S_1,S_2,\ldots,S_n)$ where 
$S_i$ is the $i$-th elementary symmetric functions in $X_1,\ldots,X_n$. 
Note $S_1=0$. The polynomials $f_n(\mathbf{S};X)$ and $f_n(\bs;X)$ are 
Tschirnhausen equivalent over $k(\mathbf{s})$, and $f_n(\mathbf{S};X)$ 
generates the field extension $k(\mathbf{X})/k(\mathbf{X})^{\frak{S}_n}$. 
By Kemper-Mattig's theorem \cite{KM00}, $f_n(\mathbf{S};X)$ 
is $k$-generic for $\frak{S}_n$ with parameters $S_2,\ldots,S_n$. 
Define 
\[
q_1:=S_n/S_{n-1},\quad q_i:=S_i/q_1^i,\ (i=2,\ldots,n-1). 
\]
Then we obtain $k(\mathbf{S})=k(q_1,\ldots,q_{n-1})$ and 
\begin{align*}
g_n(q_2,\ldots,q_{n-1};X)=(-1/q_1)^n f_n(\mathbf{S};-q_1X). 
\end{align*}
The polynomials $g_n(q_2,\ldots,q_{n-1};X)$ and $f_n(\mathbf{S};X)$ are 
Tschirnhausen equivalent over $k(\mathbf{S})$. 
Since $\mathrm{deg}(q_1)=1, \mathrm{deg}(q_i)=0, (i=2,\ldots,n-1)$, 
we also see that $g_n(q_2,\ldots,q_{n-1};X)$ is a generating polynomial 
of the degree-zero field 
$k(\mathbf{X})_0:=k(X_1/X_2,\ldots,X_{n-1}/X_n)\subset k(\mathbf{X})$ 
over $k(\mathbf{X})_0^{\frak{S}_n}=k(q_2,\ldots,q_{n-1})$, 
(cf. \cite{Kem96}, \cite[Theorem7]{KM00}). 

\begin{corollary}\label{cors2}
Let $M=k(q_2,\ldots,q_{n-1},r_2,\ldots,r_{n-1})$ be the rational function 
field with $2(n-2)$ variables. 
Let $\ba=(0,q_2,\ldots,q_{n-1},q_{n-1})\in M^n$ and 
$\bb=(0,r_2,\ldots,r_{n-1},r_{n-1})\in M^n$. 
For a fixed $j, 0\leq j \leq n-1$, assume that $F_j(\ba,\bb;X)\in M[X]$ 
has no multiple root. 
Then $F_j(\ba,\bb;X)$ is a $k$-generic polynomial for $\frak{S}_n\times \frak{S}_n$ 
with $2(n-2)$ parameters $q_2,\ldots,q_{n-1}, r_2,\ldots,r_{n-1}$. 
\end{corollary}

In order to obtain an answer to the subfield problem of generic polynomials, we study the 
degrees of the fields of Tschirnhausen coefficients $M(c_0^g,\ldots,c_{n-1}^g)$ over $M$ 
for $g\in G_{\bs,\bt}$. 
The factorization pattern of the polynomial $F_i(\ba,\bb;X)$ over $M$ gives us information 
about the degeneration of Galois groups under the specialization $(\bs, \bt) \mapsto (\ba, \bb)$ 
and about the intersection of root fields of $f_n(\ba;X)$ and $f_n(\bb;X)$ over $M$ through 
the degrees of $M(c_0^g,\ldots,c_{n-1}^g)$ over $M$ as in Proposition \ref{propc}. 

\begin{proposition}\label{propsplit}
Assume that the characteristic of $M$ is not equal to $2$ $($resp. is equal to $2)$. 
If $\Delta_\ba/\Delta_\bb\in M$ $($resp. $\beta_\ba+\beta_\bb\in M)$, then the polynomial 
$F_i(\ba,\bb;X)$ splits into two factors of degree $n!/2$ over $M$ which are not necessary 
irreducible. 
\end{proposition}
\begin{proof}
The assertion follows from Proposition \ref{prop-decom}. 
\end{proof}
\begin{corollary}
If $G_\ba,G_\bb\subset \frak{A}_n$, then $F_i(\ba,\bb;X)$ splits into two 
factors of degree $n!/2$ over $M$ which are not necessary irreducible. 
\end{corollary}
%

\section{Cubic case}\label{seCubic}

In this section, we treat the cubic case of the subfield problem of generic 
polynomial via general Tschirnhausen transformations. 
We take 
\[
f_3(\bs;X):=X^3-s_1X^2+s_2X-s_3\in k(\bs)[X] 
\]
where
\begin{align*}
s_1&=x_1+x_2+x_3,\\
s_2&=x_1x_2+x_1x_3+x_2x_3,\\
s_3&=x_1x_2x_3.
\end{align*}
As in the previous section, we take the discriminant $\Dels$ in general and 
especially the Berlekamp discriminant $\betas$ if char $k=2$: 
\begin{align}
\Dels&:=(x_2-x_1)(x_3-x_1)(x_3-x_2),\nonumber\\
\betas&:=\frac{x_1}{x_1+x_2}+\frac{x_1}{x_1+x_3}
+\frac{x_2}{x_2+x_3}\label{Berle}\\
&\,\, =\frac{x_1^2x_2+x_2^2x_3+x_1x_3^2+x_1x_2x_3}
{x_1^2x_2+x_2^2x_3+x_3^2x_1+x_1x_2^2+x_2x_3^2+x_3x_1^2}.\nonumber
\end{align}
Then we have 
\begin{align*}
\Dels^2&=s_1^2s_2^2-4s_2^3-4s_1^3s_3+18s_1s_2s_3-27s_3^2,\\
\betas(\betas+1)&=
\frac{s_2^3+s_1^3s_3+s_1s_2s_3+s_3^2}{s_1^2s_2^2+s_3^2}. 
\end{align*}
The field $k(\bs)(\Dels)$ (resp. $k(\bs)(\betas)$) is a quadratic 
extension of $k(\bs)$ if char $k\neq 2$ (resp. char $k=2$). 
For $\overline{g}=\overline{(1,\tau)}\in H\backslash \Gst$, we put 
\begin{align*}
\left(
\begin{array}{c}
u_0^g\\ u_1^g\\ u_2^g
\end{array}\right):=
\left(
\begin{array}{ccc}
1 & x_1 & x_1^2 \\ 
1 & x_2 & x_2^2 \\
1 & x_3 & x_3^2 \end{array}\right)^{-1}
\left(
\begin{array}{c}
y_1^\tau\\ y_2^\tau\\ y_3^\tau\end{array}\right). 
\end{align*}
By the definition we evaluate $(u_0,u_1,u_2)$ as 
\begin{align*}
u_0\ &=\ \Dels^{-1}\cdot\mathrm{det}
\left(\begin{array}{ccc}
y_1 & x_1 & x_1^2\\ 
y_2 & x_2 & x_2^2\\
y_3 & x_3 & x_3^2
\end{array}\right)\\
\ &=\ \frac{x_2x_3y_1}{(x_2-x_1)(x_3-x_1)}
-\frac{x_1x_3y_2}{(x_2-x_1)(x_3-x_2)}
+\frac{x_1x_2y_3}{(x_3-x_1)(x_3-x_2)},\\
u_1\ &=\ \Dels^{-1}\cdot\mathrm{det}
\left(\begin{array}{ccc}
1 & y_1 & x_1^2\\ 
1 & y_2 & x_2^2\\
1 & y_3 & x_3^2
\end{array}\right)\\
\ &=\ -\frac{(x_2+x_3)y_1}{(x_2-x_1)(x_3-x_1)}
+\frac{(x_1+x_3)y_2}{(x_2-x_1)(x_3-x_2)}
-\frac{(x_1+x_2)y_3}{(x_3-x_1)(x_3-x_2)},\\
u_2\ &=\ \Dels^{-1}\cdot\mathrm{det}
\left(\begin{array}{ccc}
1 & x_1 & y_1\\ 
1 & x_2 & y_2\\
1 & x_3 & y_3
\end{array}\right)\\ 
\ &=\ \frac{y_1}{(x_2-x_1)(x_3-x_1)}
-\frac{y_2}{(x_2-x_1)(x_3-x_2)}
+\frac{y_3}{(x_3-x_1)(x_3-x_2)}.
\end{align*}
A general form of Tschirnhausen transformations of $f_3(\bs;X)$ is 
given by 
\begin{align*}
&g_3(\bs,u_0,u_1,u_2;X)\\
&:=\ \mathrm{Resultant}_Y 
\bigl(f_3(\bs;Y),X-(u_0+u_1Y+u_2Y^2)\bigr)\\
&\ =\ X^3+(-3u_0-s_1u_1-s_1^2u_2+2s_2u_2)X^2+(3u_0^2+2s_1u_0u_1+s_2u_1^2\\
&\qquad +2s_1^2u_0u_2-4s_2u_0u_2+s_1s_2u_1u_2-3s_3u_1u_2+s_2^2u_2^2-2s_1s_3u_2^2)X\\
&\qquad -u_0^3-s_1u_0^2u_1-s_2u_0u_1^2-s_3u_1^3-s_1^2u_0^2u_2+2s_2u_0^2u_2
-s_1s_2u_0u_1u_2\\
&\qquad\ +3s_3u_0u_1u_2-s_1s_3u_1^2u_2-s_2^2u_0u_2^2+2s_1s_3u_0u_2^2-s_2s_3u_1u_2^2
-s_3^2u_2^3.
\end{align*}
By the definition, the elements $u_0,u_1,u_2$ satisfy 
\begin{align}
f_3(\bt;X)
=g_3(\bs,u_0^g,u_1^g,u_2^g;X)\ \ 
\mathrm{for}\ \ g\in \Gst.\label{eqfg}
\end{align}
We put 
\begin{align}
\As\,& {}:=\, s_1^2 - 3s_2,\nonumber\\ 
\Bs\,& {}:=\, 2s_1^3-9s_1s_2+27s_3,\label{defstd}\\
\Cs\,& {}:=\, s_1^4 - 4s_1^2s_2 + s_2^2 + 6s_1s_3,\nonumber\\
\Ds\,& {}:=\, \mathrm{Disc}_X f_3(\bs;X)=
s_1^2s_2^2-4s_2^3-4s_1^3s_3+18s_1s_2s_3-27s_3^2\quad (= \Dels^2).\nonumber
\end{align}
We can check the equality 
\begin{align}
4\As^3-\Bs^2\, =\, 27\Ds \label{eq1}
\end{align}
by direct calculation.

With the aid of computer algebra, we get the sextic polynomials 
\begin{equation*}
F_i(\bs,\bt;X)=\prod_{\overline{g}\in H\backslash \Gst}(X-u_i^g)\in K[X], \quad (i=1,2)
\end{equation*} 
as follows: 
\begin{align}
F_1(\bs,\bt;X)\, :=\ &X^6 -\frac{2\At\Cs}{\Ds} X^4
-\frac{(s_1s_2-s_3)\Bt}{\Ds} X^3 +\frac{\At^2\Cs^2}{\Ds^2} X^2\nonumber\\
&\quad \ +\frac{(s_1s_2-s_3)\At \Bt\Cs}{\Ds^2}X 
+\frac{(s_1s_2-s_3)^2\At^3 \Ds - \Cs^3 \Dt}{\Ds^3},\nonumber\\ 
F_2(\bs,\bt;X)\, :=\ &X^6 - \frac{2 \As \At}{\Ds} X^4+\frac{\Bt}{\Ds}X^3
\label{polyF2}\\
&\quad \ +\frac{\As^2 \At^2}{\Ds^2} X^2-\frac{\As \At \Bt}{\Ds^2}X
+\frac{\At^3 \Ds-\As^3\Dt}{\Ds^3}, \nonumber
\end{align}
where $\As, \Bs, \Cs$ and $\Ds$ are defined in {\rm (\ref{defstd})}. 
Here we omit the explicit description of the polynomial $F_0(\bs,\bt;X)$ because 
whose roots $u_0^g$ are available from $F_1(\bs,\bt;X)$ and 
$F_2(\bs,\bt;X)$ by (\ref{eqxy}) and (\ref{u0ch3}) below. 
The discriminant of the polynomial $F_2(\bs,\bt;X)$ 
with respect to $X$ is given by 
\begin{align}
D_{\bs,\bt}\, :=\, 
\frac{\Bs^6\Dt^3(\As^3\Bt^2-27\At^3\Ds)^2}{\Ds^{15}}.\label{eqD}
\end{align}
We note that $\As^3\Bt^2-27\At^3\Ds$ is invariant under the 
action $(x_1,x_2,x_3)\leftrightarrow (y_1,y_2,y_3)$. 
Indeed, by using (\ref{eq1}), we can obtain the following symmetric representation: 
\begin{align*}
\As^3\Bt^2-27\At^3\Ds\,=\,4\As^3\At^3-27(\At^3\Ds+\As^3\Dt).
\end{align*}
In the case of char $k\neq 2$, we also see that 
\begin{align*}
\As^3\Bt^2-27\At^3\Ds\,=\,\frac{\Bs^2\Bt^2-3^6\Ds\Dt}{4}. 
\end{align*}
By Proposition \ref{prop-decom}, we have the decomposition 
\begin{align*}
F_2(\bs,\bt;X)\, =\, F_2^+(X) F_2^-(X), 
\end{align*}
where $F_2^+(X)$ and $F_2^-(X)$ are elements of $K(\Dels/\Delt)[X]$ 
and of $K(\betas+\betat)[X]$ in the case of char $k=2$. 
In the case of char $k\neq 2$, we have, by the definition (\ref{defFpm}), 
\begin{align*}
F_2^+(X)\, & =\, \prod_{\overline{(1,\tau)}\in H\backslash \Gst\atop 
\psi(\tau)\in \mathfrak{A}_3} (X-u_2^{(1,\tau)})\, 
=\, X^3-\frac{\As\At}{\Ds}X+\frac{\Bt-\Bs(\Delt/\Dels)}{2\Ds},\\
F_2^-(X)\, & =\, \, \prod_{\overline{(1,\tau)}\in H\backslash \Gst\atop 
\psi(\tau)\not\in \mathfrak{A}_3} (X-u_2^{(1,\tau)})\, 
=\, X^3-\frac{\As\At}{\Ds}X+\frac{\Bt+\Bs(\Delt/\Dels)}{2\Ds}. 
\end{align*}
If char $k=2$, we have 
\begin{align*}
F_2^+(X)\, &=\, 
X^3+\frac{\As\At}{\Ds}X+\frac{s_1\As\Bt+t_1\At\Bs+\Bs\Bt(\betas+\betat)}{\Bs\Ds},\\
F_2^-(X)\, &=\, 
X^3+\frac{\As\At}{\Ds}X+\frac{s_1\As\Bt+t_1\At\Bs+\Bs\Bt(\betas+\betat+1)}{\Bs\Ds}. 
\end{align*}

When we specialize parameters 
$(\bs,\bt)\mapsto (\ba,\bb)\in M^3\times M^3$ for a fixed field 
$M\, (\supset k)$ with infinite elements, we assume 
that $f_3(\ba;X)$ and $f_3(\bb;X)$ are separable over $M$ 
(i.e. $D_\ba\cdot \Db\neq 0$). 
We define 
\begin{align*}
L_\ba &:=\mathrm{Spl}_M f_3(\ba;X),\hspace*{9mm}
L_\bb :=\mathrm{Spl}_M f_3(\bb;X),\\
G_\ba &:=\mathrm{Gal}(L_\ba/M),\hspace*{12mm}
G_\bb :=\mathrm{Gal}(L_\bb/M)
\end{align*}
and suppose $\# G_\ba\geq \# G_\bb$. 
We also suppose that $f_3(\ba;X)$ is irreducible over $M$. 
Then $G_\ba$ is isomorphic to either $\frak{S}_3$ or $\frak{A}_3=C_3$, 
and $G_\bb$ is isomorphic to one of $\frak{S}_3,C_3,C_2$ or $\{1\}$. 
We shall give an answer to the subfield problem of $f_3(\bs;X)$ via 
$F_j(\bs,\bt;X)$. 
More precisely, we give a necessary and sufficient condition to have 
$L_\ba\supseteq L_\bb$ for $\ba,\bb\in M^3$. 

\subsection{The case of char $k\neq 3$.}
First we study the case of char $k\neq 3$. 
The case of char $k=3$ will be studied in Subsection 4.4. 

By comparing the coefficients of (\ref{eqfg}) with respect to $X$, we obtain 
\begin{align}
u_0\, =\, \frac{t_1-s_1u_1-s_1^2u_2+2s_2u_2}{3},\quad 
u_1\, =\, \frac{Q_{1,2}(\bs,\bt;u_2)}{D_{1,2}(\bs,\bt;u_2)}\label{eqxy}
\end{align}
where 
\begin{align*}
Q_{1,2}(\bs,\bt;u_2)\, &:=\, 3\As^2\Bt-\At(6\As^3-\Bs^2+2\As\Bs s_1)u_2
+6\Ds(\As^2+\Bs s_1)u_2^3,\\
D_{1,2}(\bs,\bt;u_2)\, &:=\, 3\Bs(\As\At-3\Ds u_2^2).
\end{align*}
We also have
\begin{align}
u_0\, &=\frac{(t_1-s_1^2u_2+2s_2u_2)D_{1,2}(\bs,\bt;u_2)-s_1Q_{1,2}(\bs,\bt;u_2)}
{3D_{1,2}(\bs,\bt;u_2)}.\label{eqxy2}
\end{align}
From (\ref{eqxy}) and (\ref{eqxy2}), we can directly check $K(u_0^g,u_1^g,u_2^g)=K(u_2^g)$ 
for every $g\in \Gst$ as in Proposition \ref{prop1} where $K=k(\bs,\bt)$. 
As in (\ref{eqD0}), we obtain $D_{0,2}^0(\bs,\bt)\in k[\bs,\bt]$ and 
$D_{1,2}^0(\bs,\bt)\in k[\bs,\bt]$ as follows: 
first we see that there exist those $D_{1,2}^0(\bs,\bt)$ and $h_i(\bs,\bt)\in k[\bs,\bt]$ 
which satisfy 
\begin{align*}
\frac{1}{D_{1,2}(\bs,\bt;u_2)}\,=\,
\frac{1}{3\Bs(\As\At-3\Ds u_2^2)}\,=\,
\frac{1}{D_{1,2}^0(\bs,\bt)}\sum_{i=0}^5 h_i(\bs,\bt)u_2^i. 
\end{align*}
Actually, we get, with the aid of computer algebra, 
\[
D_{1,2}^0(\bs,\bt)\, :=\, 3\Bs(\As^3\Bt^2-27\At^3\Ds)^2
\]
and 
\begin{align*}
&\hspace*{-1.5cm}\{h_0(\bs,\bt),\ldots,h_5(\bs,\bt)\}\\
=\Big\{&4\As^2\At^2(\As^3\Bt^2 + 27\At^3\Ds - 27\Bt^2\Ds),\\
&27\Bt\Ds(4\As^3\At^3 + 9\At^3\Ds - 9\As^3\Dt),\\
&-3\As\At\Ds(5\As^3\Bt^2 + 135\At^3\Ds - 54\Bt^2\Ds),\\
&-270\As^2\At^2\Bt\Ds^2,\ 9\Ds^2(\As^3\Bt^2 + 27\At^3\Ds),\ 162\As\At\Bt\Ds^3\Big\}. 
\end{align*}
We define $D_{0,2}^0(\bs,\bt):=3\cdot D_{1,2}^0(\bs,\bt)\in k[\bs,\bt]$ on account of 
(\ref{eqxy2}). 
Hence we obtain those $P_{i,2}^0(\bs,\bt;X)\in k[\bs,\bt][X], (i=0,1)$, which satisfy 
\begin{align*}
u_i=\frac{1}{D_{i,2}^0(\bs,\bt)}\, P_{i,2}^0(\bs,\bt;u_2),\quad \mathrm{and}\quad 
\mathrm{deg}_X(P_{i,2}^0(\bs,\bt;X))=5. 
\end{align*}

After specializing parameters $(\bs,\bt)\mapsto (\ba,\bb)\in M^3\times M^3$ with 
$\Da\Db\neq 0$, we see the following lemma: 
\begin{lemma}\label{lemA}
$(\mathrm{i})$ If $f_3(\ba;X)$ is irreducible over $M$ then $\Ba\neq 0$\,{\rm ;}\\
$(\mathrm{ii})$ If $\Aa=0$, then $f_3(\ba;X)$ and $X^3-\Ba$ are Tschirnhausen equivalent. 
Hence the Galois group of $f_3(\ba;X)$ over $M(\sqrt{-3})$ is 
cyclic of order $3$\,{\rm ;}\\
$(\mathrm{iii})$ If $\Aa=0$, then $f_3(\ba;X)$ and $Y^3-3Y-(\Ba+1/\Ba)$ are 
Tschirnhausen equivalent over $M$. 
\end{lemma}
\begin{proof}
The statements of (i) and (ii) follow from the equality
\[
3^3\cdot f_3(\bs;X)=(3X-s_1)^3-3\As(3X-s_1)-\Bs. 
\]
In the case of $X^3-\Ba=0$, we put $Y=X+1/X=X(1+X/\Ba)$ and obtain 
$Y^3-3Y-(\Ba+1/\Ba)=0$. 
Thus we have (iii). 
\end{proof}

From Lemma \ref{lemA} (i), the assumption that $f_3(\ba;X)$ is irreducible 
over $M$ implies $B_\ba\neq 0$. 
From Lemma \ref{lemA} (iii), we may assume that $A_\ba\neq 0$ and 
$A_\bb\neq 0$ without loss of generality. 

By (\ref{eqD}) and the assumptions $\Ba\neq 0$ and $\Db\neq 0$, we see 
that the polynomial $F_2(\ba,\bb;X)$ has multiple roots if and only if 
$\Aa^3\Bb^2-27\Ab^3\Da=0$. 
\begin{lemma}\label{lemex}
Assume that $f_n(\ba;X)$ is irreducible and $\Aa\Ab\neq 0$. \\
$(\mathrm{i})$ If $\Aa^3\Bb^2-27\Ab^3\Da=0$ then $F_2(\ba,\bb;X)$ splits into the 
following form over $M$\,$:$
\begin{align*}
F_2(\ba,\bb;X)=
\Bigl(X-\frac{3\Ab^2}{\Aa\Bb}\Bigr)^2\Bigl(X+\frac{6\Ab^2}{\Aa\Bb}\Bigr)
\Bigl(X^3-\frac{27\Ab^4X}{\Aa^2\Bb^2}-\frac{27\Ab^3(2\Ab^3-\Bb^2)}{\Aa^3\Bb^3}\Bigr)
\end{align*}
with a simple root $-6\Ab^2/(\Aa\Bb)$\,{\rm ;}\\
$(\mathrm{ii})$ If $\Aa^3\Bb^2-27\Ab^3\Da=0$ then multiple roots 
$c=3\Ab^2/(\Aa\Bb)$ of $F_2(\ba,\bb;X)$ satisfy $\Aa\Ab-3\Da c^2=0$. 
Conversely if a root $c$ of $F_2(\ba,\bb;X)$ satisfies 
$\Aa\Ab-3\Da c^2=0$, then $\Aa^3\Bb^2-27\Ab^3\Da=0$. 
\end{lemma}
\begin{proof}
(i) We have $\Bb\neq 0$ since $27\Ab^3\Da\neq 0$. 
Using $\Da=\Aa^3\Bb^2/(27\Ab^3)$ and $\Db=(4\Ab^3-\Bb^2)/27$, 
we eliminate $\Da$ and $\Db$ from $F_2(\ba,\bb;X)$ via (\ref{polyF2}). 
Then we obtain explicit factors of $F_2(\ba,\bb;X)$ as 
\begin{align*}
&F_2(\ba,\bb;X)\\
&=X^6-\frac{54\Ab^4X^4}{\Aa^2\Bb^2}+\frac{27\Ab^3X^3}{\Aa^3\Bb}
+\frac{729\Ab^8 X^2}{\Aa^4\Bb^4}-\frac{729\Ab^7X}{\Aa^5\Bb^3}
-\frac{1458\Ab^9(2\Ab^3-\Bb^2)}{\Aa^6\Bb^6}\\
&=\Bigl(X-\frac{3\Ab^2}{\Aa\Bb}\Bigr)^2\Bigl(X+\frac{6\Ab^2}{\Aa\Bb}\Bigr)
\Bigl(X^3-\frac{27\Ab^4X}{\Aa^2\Bb^2}-\frac{27\Ab^3(2\Ab^3-\Bb^2)}{\Aa^3\Bb^3}\Bigr).
\end{align*}
We also see that $-6\Ab^2/(\Aa\Bb)$ is a simple root of $F_2(\ba,\bb;X)$ since 
$3\Ab^2/(\Aa\Bb)\neq -6\Ab^2/(\Aa\Bb)$ and
\begin{align*}
&\mathrm{Resultant}_X\Bigl(X+\frac{6\Ab^2}{\Aa\Bb}, 
X^3-\frac{27\Ab^4X}{\Aa^2\Bb^2}-\frac{27\Ab^3(2\Ab^3-\Bb^2)}{\Aa^3\Bb^3}\Bigr)\\
&=-4\Ab^3+\Bb^2=-27\Db\neq 0.
\end{align*}
(ii) For $c=3\Ab^2/(\Aa\Bb)$, the condition $\Aa^3\Bb^2-27\Ab^3\Da=0$ implies 
\begin{align*}
\Aa\Ab-3\Da c^2=\Aa\Ab-3\Bigl(\frac{\Aa^3\Bb^2}{27\Ab^3}\Bigr)
\Bigl(\frac{3\Ab^2}{\Aa\Bb}\Bigr)^2=\Aa\Ab-\Aa\Ab=0. 
\end{align*}
Conversely if $(\Aa \Ab-3\Da c^2)=0$ then 
$D_{1,2}^0(\ba,\bb)=3\Ba(\Aa^3\Bb^2-27\Ab^3\Da)^2=0$. 
\end{proof}
We obtain a solution to the field isomorphism problem of $f_3(\bs;X)$ as 
a special case of Theorem \ref{throotf} and Corollary \ref{cor1} 
in Section \ref{seSpe}. 
\begin{theorem}\label{thiso}
Assume that $f_n(\ba;X)$ is irreducible and $\Aa\Ab\neq 0$. 
Then we have\\
$(1)$ If $\Aa^3\Bb^2-27\Ab^3\Da=0$ then 
$\mathrm{Spl}_M f_3(\ba;X)=\mathrm{Spl}_M f_3(\bb;X)$\,{\rm ;}\\
$(2)$ If $\Aa^3\Bb^2-27\Ab^3\Da\neq 0$ then the following two 
conditions are equivalent\,{\rm :}\\ 
{\rm (i)} $\mathrm{Spl}_M f_3(\ba;X)
=\mathrm{Spl}_M f_3(\bb;X)$\,{\rm ;}\\
{\rm (ii)} The sextic polynomial $F_2(\ba,\bb;X)$ has a root in $M$. 
\end{theorem}
\begin{proof}
(1) Suppose $\Aa^3\Bb^2-27\Ab^3\Da=0$; then it follows from Lemma \ref{lemex} (i) 
that $F_2(\ba,\bb;X)$ has a simple root $c_2^g=-(6\Ab^2)/(\Aa\Bb)\in M$ for some $g\in \Gst$. 
Then $D_{1,2}(\ba,\bb;c_2^g)=3\Ba(\Aa\Ab-3\Da(c_2^g)^2)=-9\Ba\Aa\Ab\neq 0$. 
Thus, by (\ref{eqxy}) and (\ref{eqxy2}), we have $M(c_0^g,c_1^g,c_2^g)=M(c_2^g)=M$, 
and hence the assertion follows. 

(2) If $\Aa^3\Bb^2-27\Ab^3\Da\neq 0$ then 
$D_{0,2}^0(\ba,\bb)\cdot D_{1,2}^0(\ba,\bb)\neq 0$. 
Thus the assertion follows from Theorem \ref{throotf} (or Corollary \ref{cor1}). 
\end{proof}
\begin{example}
We give two examples which satisfy $\Aa^3\Bb^2-27\Ab^3\Da=0$. 

(1) We take $M=\mathbb{Q}$ and $\ba=(0,3,-2)$, $\bb=(3,-3,3)\in M$. 
Then 
\begin{align*}
f_3(\ba;X)=X^3+3X+2,\quad f_3(\bb;X)=X^3-3X^2-3X-3.
\end{align*}
We have $(\Aa,\Ba,\Ca,\Da)=(-9,-54,9,-216)$ and 
$(\Ab,\Bb,\Db)=(18,216,-864)$. 
The Galois group of $f_3(\ba;X)$ and $f_3(\bb;X)$ over $\mathbb{Q}$ are 
isomorphic to $\mathfrak{S}_3$ because $f_3(\ba;X)$ and $f_3(\bb;X)$ 
are irreducible over $\mathbb{Q}$, and neither of $\Da=-2^3\cdot 3^3$ 
and $\Db=-2^5\cdot 3^3$ is a square in $M$. 
We also have $\Aa^3\Bb^2-27\Ab^3\Da=0$. 
By Lemma \ref{lemex} (i) we have 
\begin{align*}
F_2(\ba,\bb;X)=\Bigl(X+\frac{1}{2}\Bigr)^2\Bigl(X-1\Bigr)
\Bigl(X^3-\frac{3X}{4}-\frac{3}{4}\Bigl). 
\end{align*}
Hence we take $c_2^g=1$ and get $(c_0^g,c_1^g,c_2^g)=(3,-1,1)$ 
by (\ref{eqxy}) and (\ref{eqxy2}). 
Thus we see that 
$\mathbb{Q}[X]/f_3(\ba;X)\cong_\mathbb{Q}\mathbb{Q}[X]/f_3(\bb;X)$.  
An explicit Tschirnhausen transformation from $f_3(\ba;X)$ to 
$f_3(\bb;X)$ over $\mathbb{Q}$ is given by 
\begin{align*}
f_3(\bb;Y)=\mathrm{Resultant}_X(f_3(\ba;X),Y-(3-X+X^2)). 
\end{align*}
We also obtain 
\begin{align*}
F_1(\ba,\bb;X)\, &=\, \Bigl(X^2-X+\frac{7}{4}\Bigr)\Bigl(X+1\Bigr)
\Bigl(X^3+\frac{3X}{4}+\frac{1}{4}\Bigr),\\
F_0(\ba,\bb;X)\, &=\, X^2(X-3)(X^3-3X^2-4). 
\end{align*}

(2) Take $M=\mathbb{Q}$ and $\ba=(-3,-4,-1)$, $\bb=(-1,-2,1)\in M$. Then 
\begin{align*}
f_3(\ba;X)=X^3+3X^2-4X+1,\quad f_3(\bb;X)=X^3+X^2-2X-1.
\end{align*}
We have $(\Aa,\Ba,\Ca,\Da)=(21,-189,259,49)$ and $(\Ab,\Bb,\Db)=(7,7,49)$. 
The Galois group of $f_3(\ba;X)$ and $f_3(\bb;X)$ over $\mathbb{Q}$ are 
isomorphic to $C_3$ because $f_3(\ba;X)$ and $f_3(\bb;X)$ are 
irreducible over $\mathbb{Q}$ and $\Da=\Db=7^2$. 
Since $\Aa^3\Bb^2-27\Ab^3\Da=0$, we have, by Lemma \ref{lemex} (i), 
\begin{align*}
F_2(\ba,\bb;X)=\Bigl(X-1\Bigr)^2\Bigl(X+2\Bigr)
\Bigl(X^3-3X-\frac{13}{7}\Bigl). 
\end{align*}
Thus we have $c_2^g=-2$. 
We obtain $(c_0^g,c_1^g,c_2^g)=(4,-7,-2)$ by (\ref{eqxy}) and (\ref{eqxy2}), and 
$\mathbb{Q}[X]/f_3(\ba;X)\cong_\mathbb{Q}\mathbb{Q}[X]/f_3(\bb;X)$. 
An explicit Tschirnhausen transformation from $f_3(\ba;X)$ to 
$f_3(\bb;X)$ over $\mathbb{Q}$ is given by 
\begin{align*}
f_3(\bb;Y)=\mathrm{Resultant}_X(f_3(\ba;X),Y-(4-7X-2X^2)). 
\end{align*}
We also obtain
\begin{align*}
F_1(\ba,\bb;X)\, &=\, \Bigl(X-3\Bigr)\Bigl(X-4\Bigr)\Bigl(X+7\Bigr)
\Bigl(X^3-37X-\frac{601}{7}\Bigr),\\
F_0(\ba,\bb;X)\, &=\, \Bigl(X+3\Bigr)\Bigl(X+2\Bigr)\Bigl(X-4\Bigr)
\Bigl(X^3+X^2-14X+\frac{71}{7}\Bigr). 
\end{align*}
Hence we have the other two Tschirnhausen transformations from 
$f_3(\ba;X)$ to $f_3(\bb;X)$ over $\mathbb{Q}$ by (\ref{eqxy}) as 
\begin{align*}
f_3(\bb;Y)=\mathrm{Resultant}_X(f_3(\ba;X),Y-(-3+3X+X^2)),\\
f_3(\bb;Y)=\mathrm{Resultant}_X(f_3(\ba;X),Y-(-2+4X+X^2)). 
\end{align*}
\end{example}

\begin{theorem}\label{th}
For $\ba,\bb\in M^3$, we assume that 
$f_n(\ba;X)$ is irreducible, $\Aa\Ab\neq 0$ and $\Aa^3\Bb^2-27\Ab^3\Da\neq 0$.
The decomposition type of irreducible factors 
$h_{\mu}(X)$ of $F_2(\ba,\bb;X)$ over $M$ gives an answer to 
the subfield problem of $f_3(\bs;X)$ as on Table $1$. 
Moreover a root field $M_{\mu}$ of each $h_{\mu}(X)$ satisfies 
$\mathrm{Spl}_{M_{\mu}} f_3(\ba,X)=\mathrm{Spl}_{M_{\mu}} f_3(\bb,X)$.
\begin{center}
{\rm Table} $1$\vspace*{4mm}\\
\begin{tabular}{|c|c|l|l|}\hline
$G_\ba$& $G_\bb$ & & $(d_{\mu}), d_{\mu}=\mathrm{deg}(h_{\mu}(X))$\\ \hline 
& & $L_\ba\neq L_\bb, 
L_\ba\cap L_\bb=M$ & $(6)$\\ \cline{3-4} 
& $\frak{S}_3$ & $L_\ba\neq L_\bb, 
[L_\ba\cap L_\bb : M]=2$ & $(3)(3)$\\ \cline{3-4}
& & $L_\ba=L_\bb$ & $(1)(2)(3)$\\ \cline{2-4}
$\frak{S}_3$ & $C_3$ & $L_\ba\cap L_\bb=M$ & $(6)$\\ \cline{2-4}
& \raisebox{-1.5ex}[0cm][0cm]{$C_2$} & $L_\ba\not\supset L_\bb$ 
& $(6)$\\ \cline{3-4}
& & $L_\ba\supset L_\bb$ & $(3)(3)$\\ \cline{2-4}
& $\{1\}$ & $L_\ba\supset L_\bb$ & $(6)$\\ \hline
& \raisebox{-1.6ex}[0cm][0cm]{$C_3$} & $L_\ba\neq L_\bb$ 
& $(3)(3)$\\ \cline{3-4}
\raisebox{-1.5ex}[0cm][0cm]{$C_3$}
& & $L_\ba=L_\bb$ & $(1)(1)(1)(3)$\\ \cline{2-4}
& $C_2$ & $L_\ba\cap L_\bb=M$ & $(6)$\\ \cline{2-4}
& $\{1\}$ & $L_\ba\supset L_\bb$ & $(3)(3)$\\ \hline
\end{tabular}
\vspace*{4mm}
\end{center}
\end{theorem}
\begin{proof}
Form Proposition \ref{prop12} and the assumption $\Aa^3\Bb^2-27\Ab^3\Da\neq 0$, we have 
\begin{align}
L_\ba L_\bb=L_\ba\, M(c_2^g)=L_\bb\, M(c_2^g)\quad \mathrm{for}\quad g\in G_{\bs,\bt}.
\label{eqLLc}
\end{align}
Hence two polynomials $f_3(\ba;X)$ and $f_3(\bb;X)$ are Tschirnhausen 
equivalent over a root field $M_{\mu}$ of each irreducible factor 
$h_{\mu}(X)$ of $F_2(\ba,\bb;X)$. 

(i) The case of $G_\ba\cong \frak{S}_3$. \\
(i-1) If $L_\ba\cap L_\bb=M$ then it follows from (\ref{eqLLc}) that $[M(c_2^g)\, :\, M]=6$. 
Hence $F_2(\ba,\bb;X)$ is irreducible over $M$. \\
(i-2) If $[L_\ba\cap L_\bb\, :\, M]=2$ then, by Proposition \ref{propsplit}, $F_2(\ba,\bb;X)$ 
splits into two cubic factors $F_2^+(X)$ and $F_2^-(X)$ over $M$, and 
each cubic factor is irreducible over $M$ because we have 
$[M(c_2^g)\, :\, M]\geq 3$ from (\ref{eqLLc}). \\
(i-3) If $G_\bb\cong\mathfrak{S}_3$ and $L_\ba=L_\bb$ then 
$F_2(\ba,\bb;X)$ also splits into two cubic factors over $M$, and one of 
them must has a linear factor from Theorem \ref{thiso}. 
By Proposition \ref{prop12} (ii), we see that the factorization pattern of 
$F_2(\ba,\bb;X)$ is equal to $(1)(2)(3)$. 

(ii) The case of $G_\ba\cong C_3$. \\
(ii-1) If $G_\bb\cong C_3$ and $L_\ba\neq L_\bb$ then, by 
Proposition \ref{propsplit}, $F_2(\ba,\bb;X)$ splits into two cubic factors 
over $M$ and each cubic factor is irreducible over $M$ because we have 
$[M(c_2^g)\, :\, M]\geq 3$ from (\ref{eqLLc}). \\
(ii-2) If $G_\bb\cong C_3$ and $L_\ba=L_\bb$ then $F_2(\ba,\bb;X)$ also 
splits into two cubic factors over $M$, and one of them must has a linear factor from Theorem 
\ref{thiso}. 
By Proposition \ref{prop12} (ii), we see that the factorization pattern of 
$F_2(\ba,\bb;X)$ is equal to $(1)(1)(1)(3)$. \\
(ii-3) If $G_\bb\cong C_2$ then it follows from (\ref{eqLLc}) that $[M(c_2^g)\, :\, M]=6$. 
Hence $F_2(\ba,\bb;X)$ is irreducible over $M$. \\
(ii-4) If $G_\bb\cong\{1\}$ then, by Proposition \ref{propsplit}, 
$F_2(\ba,\bb;X)$ splits into two cubic factors $F_2^+(X)$ and $F_2^-(X)$ 
over $M$, and each cubic factor is irreducible over $M$ because we have 
$[M(c_2^g)\, :\, M]=3$. 
\end{proof}
%

\subsection{Special case $1 : X^3+S_2X-S_3$.}
Assume that char $k\neq 3$. 
We treat a $k$-generic polynomial of the form $f_3(0,S_2,S_3;X)=X^3+S_2X-S_3$. 
Define $\mathbf{X}:=(X_1, X_2, X_3)$, 
\[
X_1\,:=\,x_1-s_1/3,\ X_2\,:=\,x_2-s_1/3,\ X_3\,:=\,x_3-s_1/3. 
\]
Then $k(\mathbf{X}):=k(X_1,X_2,X_3)\subset k(x_1,x_2,x_3)$ and $X_1+X_2+X_3=0$. 
The action of $\frak{S}_3$ on $k(x_1,x_2,x_3)$ induces an action on 
$k(\mathbf{X})$ which is linear and faithful. 
We see $k(\mathbf{X})^{\frak{S}_3}=k(\mathbf{S})$ where 
$\mathbf{S}=(S_1,S_2,S_3)$ and $S_i$ is the $i$-th elementary symmetric 
function in $X_1,X_2,X_3$. We have 
\[
S_1\,=\,0,\ S_2\,=\,-\frac{\As}{3}=\frac{-(s_1^2-3s_2)}{3},\ 
S_3\,=\,\frac{\Bs}{27}=\frac{2s_1^3-9s_1s_2+27s_3}{27}.
\]
The polynomials $f_3(0,S_2,S_3;X)$ and $f_3(\bs;X)$ are Tschirnhausen 
equivalent over $k(\mathbf{s})$. 
Moreover, $f_3(\mathbf{S};X)$ generates the field extension 
$k(\mathbf{X})/k(\mathbf{X})^{\frak{S}_3}$. 
After specializing parameters $\bs\mapsto \ba=(a_1,a_2,a_3)\in M^3$, the polynomials 
$f_3(\ba;X)=X^3-a_1X^2+a_2X-a_3$ and 
$f_3(0,A_2,A_3;X)=X^3+A_2X-A_3$ are Tschirnhausen equivalent 
over $M$ where $A_2:=-\Aa/3, A_3:=\Ba/27$. 
Put $\mathbf{T}:=(0,T_2,T_3)$. 
Then we have 
\begin{align*}
&D_\mathbf{S}=\mathrm{Disc}_Xf_3(0,S_2,S_3;X)=-4S_2^3-27S_3^2,\\
&A_\mathbf{S}^3B_\mathbf{T}^2-27A_\mathbf{T}^3D_\mathbf{S}
=-729(4S_2^3T_2^3+27S_3^2T_2^3+27S_2^3T_3^2). 
\end{align*}
We also obtain 
\begin{align}
F_0(\mathbf{S},\mathbf{T};X)=X^6&-\frac{8S_2^3T_2}{D_\mathbf{S}}X^4
+\frac{8S_2^3T_3}{D_\mathbf{S}}X^3\nonumber\\
&+\frac{16S_2^6T_2^2}{D_\mathbf{S}^2}X^2-\frac{32S_2^6T_2T_3}{D_\mathbf{S}^2}X
+\frac{64S_2^6(S_3^2T_2^3 - S_2^3T_3^2)}{D_\mathbf{S}^3},\nonumber\\
F_1(\mathbf{S},\mathbf{T};X)=X^6&+\frac{6S_2^2T_2}{D_\mathbf{S}}X^4
+\frac{27S_3T_3}{D_\mathbf{S}}X^3+ \frac{9S_2^4T_2^2}{D_\mathbf{S}^2}X^2
+\frac{81S_2^2S_3T_2T_3}{D_\mathbf{S}^2}X\label{polyF1ST}\\
&+\frac{4S_2^6T_2^3 + 108S_2^3S_3^2T_2^3 + 729S_3^4T_2^3 + 27S_2^6T_3^2}{D_\mathbf{S}^3},
\nonumber\\
F_2(\mathbf{S},\mathbf{T};X)=X^6&-\frac{18S_2T_2}{D_\mathbf{S}}X^4
+ \frac{27T_3}{D_\mathbf{S}}X^3\nonumber\\
&+\frac{81S_2^2T_2^2}{D_\mathbf{S}^2}X^2 - \frac{243S_2T_2T_3}{D_\mathbf{S}^2}X
+ \frac{729(S_3^2 T_2^3 - S_2^3 T_3^2)}{D_\mathbf{S}^3}.\nonumber
\end{align}
Put $F_2^0(S_2,S_3,T_2,T_3;X):=F_2(\mathbf{S},\mathbf{T};X)$. 
Then we have 
\begin{theorem}
For $(A_2,A_3), (B_2,B_3)\in M^2$ with $4A_2^3B_2^3+27A_3^2B_2^3+27A_2^3B_3^2\neq 0$, 
the decomposition type of irreducible factors $h_{\mu}(X)$ of 
$F_2^0(A_2,A_3,B_2,B_3;X)$ over $M$ gives an answer to the subfield 
problem of $X^3+S_2X-S_3$ as on Table $1$. 
\end{theorem}

\subsection{Special case $2 : X^3+sX+s$.}
Assume that char $k\neq 3$. 
Define $A_2:=-\Aa/3, A_3:=\Ba/27$ as in the previous subsection. 
For $\ba=(a_1,a_2,a_3)\in M^3$ with $\Aa\neq 0$ and $\Ba\neq 0$, 
the polynomials $f_3(\ba;X)$ and $f_3(0,a,-a;X)=X^3+aX+a$ are 
Tschirnhausen equivalent over $M$, where 
\begin{align*}
a:=\frac{A_2^3}{A_3^2}=-\frac{27\Aa^3}{\Ba^2}
=-\frac{27(a_1^2-3a_2)^3}{(2a_1^3-9a_1a_2+27a_3)^2}. 
\end{align*}
This follows from the equality 
\[
X^3+A_2X-A_3=-\frac{A_3^3}{A_2^3}\Bigl(\Bigl(-\frac{A_2X}{A_3}\Bigr)^3
+a\Bigl(-\frac{A_2X}{A_3}\Bigr)+a\Bigr). 
\]

We take $\ba=(0,a,-a)\in M^3$ and $\bb=(0,b,-b)\in M^3$. 
Then
\begin{align*}
&\Da=\mathrm{Disc}_Xf_3(0,a,-a;X)=-a^2(4a+27),\\
&A_\ba^3B_\bb^2-27A_\bb^3D_\ba=-729a^2b^2(4ab+27a+27b). 
\end{align*}

\begin{remark}
It may be notable that the Tschirnhausen equivalence in the previous two subsections is 
affinely obtained; that is, it is given by linear forms. 
\end{remark}
By Theorem \ref{thiso}, we see that if $4ab+27a+27b=0$ for $a,b\in M$ then 
$X^3+aX+a$ and that $X^3+bX+b$ are Tschirnhausen equivalent 
over $M$. 
Hence 
\[
X^3+aX+a\quad \mathrm{and}\quad X^3-\frac{27a}{4a+27}X-\frac{27a}{4a+27}
\]
have the same splitting field over $M$. 
\begin{example}
For $M=\mathbb{Q}$, we have 
\begin{align*}
\mathrm{Spl}_\mathbb{Q} (X^3-189X-189)\,&=\,\mathrm{Spl}_\mathbb{Q} (X^3-7X-7),\\
\mathrm{Spl}_\mathbb{Q} (X^3-27X-27)\,&=\,\mathrm{Spl}_\mathbb{Q} (X^3-9X-9),\\
\mathrm{Spl}_\mathbb{Q} (X^3-6X-6)\,&=\,\mathrm{Spl}_\mathbb{Q} (X^3+54X+54). 
\end{align*}
\end{example}
In our special case, we have 
\begin{align*}
F_0(0,s,-s,0,t,-t;X)&=X^6-\frac{8s^3t}{\Ds}X^4-\frac{8s^3t}{\Ds}X^3\\
&\hspace*{11mm}
+\frac{16s^6t^2}{\Ds^2}X^2+\frac{32s^6t^2}{\Ds^2}X-\frac{64s^8t^2(s - t)}{\Ds^3},\\
F_1(0,s,-s,0,t,-t;X)&=X^6+\frac{6s^2t}{\Ds}X^4+\frac{27st}{\Ds}X^3+\frac{9s^4t^2}{\Ds^2}X^2\\
&\hspace*{11mm}+\frac{81s^3t^2}{\Ds^2}X
+\frac{s^4t^2(27s^2 + 729t + 108st + 4s^2t)}{\Ds^3},\\
F_2(0,s,-s,0,t,-t;X)&=X^6-\frac{18st}{D_\bs}X^4-\frac{27t}{D_\bs}X^3\\
&\hspace*{11mm}
+\frac{81s^2t^2}{\Ds^2}X^2+\frac{243st^2}{\Ds^2}X-\frac{729s^2t^2(s-t)}{\Ds^3}
\end{align*}
where $\Ds=-s^2(4s+27)$. 
We put $G_2(s,t;X):=F_2(0,s,-s,0,t,-t;X)$, and have the following theorem. 
\begin{theorem}
For $a,b\in M$ with $4ab+27a+27b\neq 0$, the decomposition type of irreducible factors 
$h_{\mu}(X)$ of $G_2(a,b;X)$ over $M$ gives an answer to the subfield problem of 
$X^3+sX+s$ as on Table $1$. 
\end{theorem}

In the paper \cite{HM07}, we gave an answer to the field isomorphism problem of 
$X^3+sX+s$ in the case of char $k\neq 3$. 
Here we present a slightly modified version of the result in \cite{HM07} 
(cf. Theorem 1 and Theorem 7 in \cite{HM07}). 
First we note that if $G_2(a,b;X)$ has a root zero, i.e. $G_2(a,b;0)=0$, then $ab(a-b)=0$. 
Assume that $a\neq b$. 
Then we have $c_2^g\neq 0$ for $\overline{g}\in H\backslash G_{\bs,\bt}$. 
Put $u:=3c_1/c_2$. 
Then we obtain $(c_0,c_1)=(2ac_2/3, uc_2/3)$ and 
\begin{align*}
c_2=\frac{3(u^2+9u-3a)}{u^3-2au^2-9au-2a^2-27a}.
\end{align*}
Under the condition $a(4a+27)\neq 0$, we see that $u^3-2au^2-9au-2a^2-27a\neq 0$. 
Hence we have $M(c_0,c_1,c_2)=M(u)$. 
From the direct computation, we obtain $(a-b)\cdot 
\prod_{\overline{g}\in H\backslash G_{\bs,\bt}}(X-u^g)=:H(a,b;X)$ 
where 
\[
H(a,b;X)=a(X^2+9X-3a)^3-b(X^3-2aX^2-9aX-2a^2-27a)^2.
\]
Note that the polynomial $H(s,t;X)\in k(s,t)[X]$ is $k$-generic for $S_3\times S_3$. 
We also see $\mathrm{Disc}_X H(a,b;X)=a^{10}b^4(4a+27)^{15}(4b+27)^3$.
\begin{theorem}\label{thS3}
For $a,b\in M$ with $a\neq b$, the decomposition type 
of irreducible factors $h_{\mu}(X)$ of $H(a,b;X)$ over $M$ gives 
an answer to the subfield problem of $X^3+sX+s$ as on Table $1$. 
In particular, two splitting fields of $X^3+aX+a$ and of $X^3+bX+b$ 
over $M$ coincide if and only if there exists $u\in M$ such that 
\[
b=\frac{a(u^2+9u-3a)^3}{(u^3-2au^2-9au-2a^2-27a)^2}.
\]
\end{theorem}
By applying Hilbert's irreducibility theorem and Siegel's theorem for curves of genus $0$ 
to Theorem \ref{thS3} respectively, we obtain the following corollaries: 
\begin{corollary}\label{Hil1}
If $M$ is a Hilbertian field then for a fixed $a\in M$ there exist infinitely many 
$b\in M$ such that $\mathrm{Spl}_M (X^3+aX+a)=\mathrm{Spl}_M (X^3+bX+b)$. 
\end{corollary}
\begin{corollary}\label{Siegel1}
Let $M$ be a number field and $\mathcal{O}_M$ the ring of integers in $M$. 
For a given integer $a\in \mathcal{O}_M$, there exist only finitely many integers 
$b\in\mathcal{O}_M$ such that $\mathrm{Spl}_{M} (X^3+aX+a)=\mathrm{Spl}_M (X^3+bX+b)$. 
\end{corollary}
\begin{proof}
We may apply Siegel's theorem (cf. \cite[Theorem 6.1]{Lan78}, 
\cite[Chapter 8, Section 5]{Lan83}, \cite[Theorem D.8.4]{HS00}) 
to Theorem \ref{thS3} because the discriminant of $u^3-2au^2-9au-2a^2-27a$ 
with respect to $u$ equals $-a^2(4a+27)^3$. 
\end{proof}
\begin{remark}
T. Komatsu \cite{Kom} treated a cubic generic polynomial $g(t,Y)=Y^3-t(Y+1)\in k(t)[Y]$ 
in the case of char $k\neq 2, 3$. 
He obtained a sextic polynomial $P(t_1,t_2;Z)$ satisfying $\mathrm{Spl}_{k(t_1,t_2)} 
P(t_1,t_2;Z)=\mathrm{Spl}_{k(t_1,t_2)}\, g(t_1,Y)\cdot \mathrm{Spl}_{k(t_1,t_2)}\, g(t_2,Y)$ 
via his descent Kummer theory (see also \cite{Kom04}). 
His paper \cite{Kom} treats the subfield problem of $g(t,Y)$ by using the sextic polynomial 
$P(t_1,t_2;Z)$. 
\end{remark}

\subsection{The case of char $k=3$}\label{subsech3}
In this subsection we study the case of char $k=3$. 
We have, in this case, 
\begin{align*}
\As=s_1^2,\quad \Bs=-s_1^3,\quad \Ds=s_1^2s_2^2-s_2^3-s_1^3s_3.
\end{align*}
By comparing the coefficients of (\ref{eqfg}) with respect to $X$, we obtain
\begin{align}
u_0&=\frac{s_2t_1^2-s_1^2t_2-s_2t_1(s_1^2-s_2)u_2-\Ds u_2^2}{s_1^2 t_1},\label{u0ch3}\\
u_1&=\frac{t_1-(s_1^2+s_2)u_2}{s_1}.\nonumber
\end{align}
We also get 
\begin{align*}
F_2(\bs,\bt;X)=X^6+\frac{s_1^2t_1^2}{\Ds}X^4-\frac{t_1^3}{\Ds}X^3
+\frac{s_1^4t_1^4}{\Ds^2}X^2+\frac{s_1^2t_1^5}{\Ds^2}X+\frac{t_1^6\Ds-s_1^6\Dt}{\Ds^3}. 
\end{align*}
The discriminant of $F_2(\bs,\bt;X)$ with respect to $X$ is given by 
\[
D_{\bs,\bt}\, =\, 
\frac{\Bs^6\Dt^3(\As^3\Bt^2-27\At^3\Ds)^2}{\Ds^{15}}
=\frac{s_1^{18}\Dt^3(s_1^6t_1^6)^2}{\Ds^{15}}
=\frac{s_1^{30}t_1^{12}\Dt^3}{\Ds^{15}}.
\]
Hence we have
\begin{theorem}
For $\ba=(a_1,a_2,a_3), \bb=(b_1,b_2,b_3)\in M^3$ with $a_1b_1\neq 0$, 
the decomposition type of irreducible factors $h_{\mu}(X)$ of 
$F_2(\ba,\bb;X)$ over $M$ gives an answer to the subfield problem of 
$X^3-s_1X^2+s_2X-s_3$ as on Table $1$. 
\end{theorem}

Next we treat the case of $a_1=0$. 
Without loss of generality we may assume that $a_1=0, b_1=0$ 
because the polynomial $f_3(0,s,-s)=X^3+sX+s$ is $k$-generic 
for $\mathfrak{S}_3$. 
Actually for $f_3(\bs;X)=X^3-s_1X^2+s_2X-s_3=0$, we take 
\[
Y=\frac{s_1^2}{-s_2-s_1X}
\]
and have
\[
Y^3+\frac{-s_1^6}{s_1^2s_2^2-s_2^3-s_1^3s_3}Y+\frac{-s_1^6}{s_1^2s_2^2-s_2^3-s_1^3s_3}=0. 
\]
Hence for $\ba=(a_1,a_2,a_3)\in M^3$ with $a_1\cdot \Da\neq 0$, 
the polynomials $f_3(\ba;X)$ and 
\[
X^3+\frac{-a_1^6}{a_1^2a_2^2-a_2^3-a_1^3a_3}X+\frac{-a_1^6}{a_1^2a_2^2-a_2^3-a_1^3a_3}
\]
are Tschirnhausen equivalent over $M$. 

After specializing parameters $(s_1,t_1)=(0,0)$ and then 
by comparing the coefficients of (\ref{eqfg}) with respect to $X$, 
we obtain
\begin{align*}
u_1=\frac{s_2(t_3-t_2u_0-u_0^3)}{s_3t_2},\quad u_2=0.
\end{align*}
Since the equalities (\ref{polyF1ST}) are also valid for char $k=3$, we have 
\begin{align*}
F_0(0,s_2,s_3,0,t_2,t_3;X)\,&=\,X^6-t_2X^4+t_3X^3+t_2^2X^2+t_2t_3X
+\frac{s_2^3t_3^2-s_3^2t_2^3}{s_2^3},\\
F_1(0,s_2,s_3,0,t_2,t_3;X)\,&=\,\Bigl(X^2-\frac{t_2}{s_2}\Bigr)^3,\\
F_2(0,s_2,s_3,0,t_2,t_3;X)\,&=\,X^6.
\end{align*}
We take $f_3(0,s,-s;X)=X^3+sX+s$. 
Then $\mathrm{Disc}_Xf_3(0,s,-s;X)=-s^3$. 
Define 
\begin{align*}
G_0(s,t;X)&:=F_0(0,s,-s,0,t,-t)\\
&\ =X^6-tX^4-tX^3+t^2X^2-t^2X+\frac{t^2(s-t)}{s}. 
\end{align*}
We see that the discriminant of $G_0(s,t;X)$ with respect to $X$ is equal to $t^{15}/s^3$. 

\begin{proposition}
The polynomial $G_0(s,t;X)$ is $k$-generic for 
$\mathfrak{S}_3\times \mathfrak{S}_3$. 
\end{proposition}
\begin{theorem}
For $a,b\in M$ with $ab\neq 0$, 
the decomposition type of irreducible factors $h_{\mu}(X)$ of $G_0(a,b;X)$  
over $M$ gives an answer to the subfield problem of $X^3+sX+s$ as 
on Table $1$. 
\end{theorem}
%

\section{Cyclic cubic case}\label{seCyc}

Let the cyclic substitution $\sigma=(123)\in \mathfrak{S}_3$ act on $k(x_1,x_2,x_3)$ as 
\[
\sigma\ :\ x_1\ \longmapsto\ x_2,\ x_2\ \longmapsto\ x_3,\ x_3\ \longmapsto\ x_1. 
\]
Put $K_1=k(z_1,z_2,z_3)\subset k(x_1,x_2,x_3)$ where 
\[
z_1:=\frac{x_1-x_2}{x_2-x_3},\quad z_2:=\frac{x_2-x_3}{x_3-x_1},\quad 
z_3:=\frac{x_3-x_1}{x_1-x_2}.
\]
Then we have 
\[
z_2=\frac{-1}{1+z_1},\quad z_3=\frac{-(1+z_1)}{z_1}. 
\]
Hence the transcendental degree of $K_1$ over $k$ is equal to one, and 
$C_3=\langle\sigma\rangle$ acts on $K_1=k(z_1)$ faithfully as 
\[
\sigma\ :\ z_1\ \longmapsto\ \frac{-1}{1+z_1}\ \longmapsto\ \frac{-(1+z_1)}{z_1}\ 
\longmapsto\ z_1. 
\]
We consider the $C_3$-extension $K_1/K_1^{C_3}$. 
Put 
\begin{align*}
g^{C_3}(\widetilde{m};X)&=\prod_{x\in \mathrm{Orb}_{\langle\sigma\rangle}(z_1)}(X-x)
=\Bigl(X-z_1\Bigr)\Bigl(X+\frac{1}{1+z_1}\Bigr)\Bigl(X+\frac{1+z_1}{z_1}\Bigr)\\
&=X^3-\widetilde{m}X^2-(\widetilde{m}+3)X-1,
\end{align*}
where 
\begin{align*}
\widetilde{m}=\frac{z_1^3-3z_1-1}{z_1(z_1+1)}=
\frac{-(x_1^3+x_2^3+x_3^3-3x_1^2x_2-3x_2^2x_3-3x_3^2x_1+6x_1x_2x_3)}
{(x_2-x_1)(x_3-x_1)(x_3-x_2)}.
\end{align*}
We have $K_1^{C_3}=k(\widetilde{m})$ and 
the splitting field of $g^{C_3}(\widetilde{m};X)$ over $k(\widetilde{m})$ equals $K_1$. 
The polynomial $g^{C_3}(\widetilde{m};X)$ is $k$-generic for $C_3$ and is well-known 
as Shanks' simplest cubic \cite{Sha74}.
\begin{lemma}\label{lemC3}
The polynomials $f_3(\bs;X)=X^3-s_1X^2+s_2X-s_3$ and 
$g^{C_3}(\widetilde{m};X)=X^3-\widetilde{m}X^2-(\widetilde{m}+3)X-1$ 
are Tschirnhausen equivalent over $k(x_1,x_2,x_3)^{C_3}$. 
\end{lemma}
\begin{proof} 
We see that the splitting field of $g^{C_3}(\widetilde{m};X)$ over 
$k(x_1,x_2,x_3)^{C_3}$ is $k(x_1,x_2,x_3)$ as follows: 

(i) The case of char $k\neq 2$. 
We see $k(x_1,x_2,x_3)^{C_3}=k(s_1,s_2,s_3,\Dels)$ 
where $\Dels=(x_2-x_1)(x_3-x_1)(x_3-x_2)$. 
We have, furthermore, 
\begin{align}
x_i=\frac{-\Dels+s_1s_2-9s_3-2\Dels z_i}{2(s_1^2-3s_2)},\quad 
z_i=-\frac{\Dels-s_1s_2+9s_3}{2\Dels}-\frac{(s_1^2-3s_2)x_i}{\Dels}\label{eqc3xz}
\end{align}
for $i=1,2,3$. 

(ii) The case of char $k=2$. 
We now see $k(x_1,x_2,x_3)^{C_3}=k(s_1,s_2,s_3,\betas)$ where $\betas$ 
is the Berlekamp discriminant as in (\ref{Berle}). 
We have, for $i=1,2,3$, 
\[
x_i=\frac{(s_1s_2+s_3)(\betas +z_i)}{s_1^2+s_2},\quad 
z_i=\betas+\frac{(s_1^2+s_2)x_i}{s_1s_2+s_3}. 
\]
\end{proof}

We express the element $\widetilde{m}$ in terms of $s_1,s_2,s_3$ and 
$\Dels$ (resp. $\betas$) in the case of char $k\neq 2$ (resp. char $k=2$). 
For an arbitrary field $k$, we have 
\[
k(x_1,x_2,x_3)^{C_3}=k(s_1,s_2,s_3,x_1x_2^2+x_2x_3^2+x_3x_1^2)
\]
and 
\[
\widetilde{m}=\frac{-s_1^3+6s_1s_2-18s_3-3(x_1x_2^2+x_1^2x_3+x_2x_3^2)}
{-s_1s_2+3s_3+2(x_1x_2^2+x_1^2x_3+x_2x_3^2)}. 
\]
Then in the case of char $k\neq 2$, it follows from 
\begin{align*}
x_1x_2^2+x_2x_3^2+x_3x_1^2=(\Dels+s_1s_2-3s_3)/2 
\end{align*}
that 
\begin{align*}
\qquad \widetilde{m}=-\frac{3\Dels+2s_1^3-9s_1s_2+27s_3}{2\Dels}\quad 
\Bigl(=\ -\frac{3\Dels+\Bs}{2\Dels}\Bigr). 
\end{align*}
In the case of char $k=2$, we have 
\begin{align*}
x_1x_2^2+x_2x_3^2+x_3x_1^2=s_1s_2+\betas s_1s_2+\betas s_3,
\end{align*}
and then 
\begin{align*}
\qquad \widetilde{m}=\frac{s_1^3+s_1s_2+\betas s_1s_2+\betas s_3}{s_1s_2+s_3}\quad 
\Bigl(=\ \frac{s_1\As+\betas \Bs}{\Bs}\Bigr). 
\end{align*}

By applying Theorem \ref{th} to $M=k(x_1,x_2,x_3)^{C_3}$, there exist 
three Tschirnhausen transformations from $f_3(\bs;X)$ to 
$g^{C_3}(\widetilde{m};X)$ which are defined over $k(x_1,x_2,x_3)^{C_3}$. 
By specializing parameters $(t_1,t_2,t_3)\mapsto (\widetilde{m},-(\widetilde{m}+3),1)\in 
k(x_1,x_2,x_3)^{C_3}$ of the polynomial $F_2(\bs,\bt;X)$, we explicitly obtain coefficients 
$(c_0^g,c_1^g,c_2^g)$ of Tschirnhausen transformations from $f_3(\bs;X)$ to 
$g^{C_3}(\widetilde{m};X)$ over $k(x_1,x_2,x_3)^{C_3}$: 
\[
g^{C_3}(\widetilde{m};X)=\mathrm{Resultant}_Y\bigl(f_3(\bs;Y),X-(c_0^g+c_1^gY+c_2^gY^2)\bigr). 
\]
In the case of char $k\neq 2$, we obtain explicit factors of 
$F_2(\bs,\widetilde{m},-(\widetilde{m}+3),1;X)$ by (\ref{polyF2}): 
\begin{align*}
&F_2(s_1,s_2,s_3,\widetilde{m},-(\widetilde{m}+3),1;X)\\
&=X\Bigl(X-\frac{\As^2}{\Dels^2}\Bigr)\Bigl(X+\frac{\As^2}{\Dels^2}\Bigr)
\Bigl(X^3-\frac{\As^4}{\Dels^4}X-\frac{\As^3(2\As s_1-3s_1s_2+27s_3)}{\Dels^5}\Bigr).
\end{align*}
Hence we have 
$\bigl\{c_2^g\ |\ \overline{g}=\overline{(1,\tau)}\in H\backslash G_{\bs,\bt},\ \psi(\tau)\in 
\mathfrak{A}_3\bigr\}=\{0, \As^2/\Dels^2,-\As^2/\Dels^2\}$. 
From (\ref{eqc3xz}) we see $c_2=0$. 
By (\ref{eqxy}) and (\ref{eqxy2}), we obtain $c_0^g$ and $c_1^g$ 
from $c_2^g$: 
\begin{align*}
&(c_0,c_1,c_2)=\Bigl(\frac{\Es-\Dels}{2\Dels},-\frac{\As}{\Dels},0\Bigr),\\
&(c_0^{g_1},c_1^{g_1},c_2^{g_1})=\\
&\Bigl(\frac{\As(\As s_2-s_2^2+3s_1s_3)-(\As s_1-\Es)\Dels-\Dels^2}{2\Dels^2},
\frac{\As(-2\As s_1+\Es+\Dels)}{2\Dels^2},\frac{\As^2}{\Dels^2}\Bigr),\\
&(c_0^{g_2},c_1^{g_2},c_2^{g_2})=\\
&\Bigl(\frac{-\As(\As s_2-s_2^2+3s_1s_3)-(\As s_1-\Es)\Dels-\Dels^2}{2\Dels^2},
\frac{\As(2\As s_1-\Es+\Dels)}{2\Dels^2},-\frac{\As^2}{\Dels^2}\Bigr)
\end{align*}
where $\Es=s_1s_2-9s_3$, 
$\overline{g_i}=\overline{(1,\tau_i)}\in H\backslash G_{\bs,\bt}$ 
and $\psi(\tau_i)\in \mathfrak{A}_3\backslash\{1\}, (i=1,2)$. 
With the aid of computer algebra, we can check 
\begin{align}
z_2=\frac{x_2-x_3}{x_3-x_1}=c_0^{g_1}+c_1^{g_1}x_1+c_2^{g_1}x_1^2,\quad 
z_3=\frac{x_3-x_1}{x_1-x_2}=c_0^{g_2}+c_1^{g_2}x_1+c_2^{g_2}x_1^2.\label{eqz2z3}
\end{align}
Hence we get $\psi(\tau_1)=(123)\in \mathfrak{A}_3$ and 
$\psi(\tau_2)=(132)\in \mathfrak{A}_3$. 

In the case of char $k=2$, we see $\As^3\Bt^2-27\At^3\Ds=0$ where 
$\bt=(t_1,t_2,t_3)=(\widetilde{m},-(\widetilde{m}+3),1)$. 
Hence, by using Lemma \ref{lemex} (i), we have 
\begin{align*}
&F_2(s_1,s_2,s_3,\widetilde{m},-(\widetilde{m}+3),1;X)\\
&=X\Bigl(X+\frac{(s_1^2+s_2)^2}{(s_1s_2+s_3)^2}\Bigr)^2
\Bigl(X^3+\frac{(s_1^2+s_2)^4}{(s_1s_2+s_3)^4}X
+\frac{(s_1^2+s_2)^3}{(s_1s_2+s_3)^4}\Bigr).
\end{align*}

By (\ref{eqxy}) and (\ref{eqxy2}) we obtain 
\[
(c_0,c_1,c_2)=\Bigl(\betas,\frac{s_1^2+s_2}{s_1s_2+s_3},0\Bigr).
\]
We also obtain the coefficients of two other Tschirnhausen 
transformations which satisfy (\ref{eqz2z3}): 
\begin{align*}
&(c_0^{g_1},c_1^{g_1},c_2^{g_1})=\\
&\Bigl(\frac{\betas(s_1^3+s_3)}{s_1s_2+s_3},
\frac{(s_1^2+s_2)(s_1^3+s_3+\betas s_1s_2+\betas s_3)}{(s_1s_2+s_3)^2},
\frac{(s_1^2+s_2)^2}{(s_1s_2+s_3)^2}\Bigr),\\
&(c_0^{g_2},c_1^{g_2},c_2^{g_2})=\\
&\Bigl(\frac{s_1^3+s_1s_2+\betas s_1^3+\betas s_3}{s_1s_2+s_3},
\frac{(s_1^2+s_2)(s_1^3+s_1s_2+\betas s_1s_2+\betas s_3)}{(s_1s_2+s_3)^2},
\frac{(s_1^2+s_2)^2}{(s_1s_2+s_3)^2}\Bigr). 
\end{align*}

Now we assume that the Galois group of $f_3(\ba;X)=X^3-a_1X^2+a_2X-a_3$ over $M$ is 
isomorphic to $C_3$. 
By specializing parameters $\bs=(s_1,s_2,s_3)\mapsto \ba=(a_1,a_2,a_3)\in M^3$ with 
$\Aa\neq 0$ and $\Ba\neq 0$ in Lemma \ref{lemC3}, we see that $f_3(\ba;X)$ and 
$g^{C_3}(m;X)=X^3-mX^2-(m+3)X-1$ are Tschirnhausen equivalent over $M$ where 
\begin{align}
m=
\begin{cases}
\, \displaystyle{-\frac{3\Dela+2a_1^3-9a_1a_2+27a_3}{2\Dela}},
\ \ \mathrm{if}\ \ \mathrm{char}\ k\neq 2,\vspace*{3mm}\\
\, \displaystyle{\frac{a_1^3+a_1a_2+\betaa a_1a_2+\betaa a_3}{a_1a_2+a_3}},\ \ 
\ \ \mathrm{if}\ \ \mathrm{char}\ k=2. 
\end{cases}\label{eqm}
\end{align}

From now on we take $\ba=(m,-(m+3),1), \bb=(n,-(n+3),1)$. 
Then we have 
$f_3(\ba;X)=X^3-mX^2-(m+3)X-1, f_3(\bb;X)=X^3-nX^2-(n+3)X-1$, and 
\begin{align*}
&\Da=\mathrm{Disc}_Xf_3(\ba;X)=(m^2+3m+9)^2,\\
&\Aa^3\Bb^2-27\Ab^3\Da=\Da\Db(2mn+3m+3n+18)(2mn+3m+3n-9). 
\end{align*}
We have $\Dela=m^2+3m+9$ and $\Delb=n^2+3n+9$ 
and also 
\begin{align*}
F_2(\ba,\bb;X)=F_2^+(m,n;X)F_2^-(m,n;X)
\end{align*}
where 
\begin{align*}
F_2^+(m,n;X)\,&=\,X^3-\frac{\Delb}{\Dela}X-\frac{(m-n)\Delb}{\Dela^2},\\
F_2^-(m,n;X)\,&=\,X^3-\frac{\Delb}{\Dela}X+\frac{(m+n+3)\Delb}{\Dela^2}
\end{align*}
and 
\begin{align*}
\mathrm{Disc}_X(F_2^+(m,n;X))&=\frac{\Delb^2(2mn+3m+3n+18)^2}{\Dela^4},\\
\mathrm{Disc}_X(F_2^-(m,n;X))&=\frac{\Delb^2(2mn+3m+3n-9)^2}{\Dela^4}.
\end{align*}
Note that $F_2^-(m,n;X)=F_2^+(m,-n-3;X)$. 
We see that if $m+n+3=0$ then $X^3-mX^2-(m+3)X-1$ and 
$X^3-nX^2-(n+3)X-1$ have the same splitting field over $M$. 
Hence 
\begin{align*}
X^3-mX^2-(m+3)X-1\quad \mathrm{and}\quad X^3+(m+3)X^2+mX-1
\end{align*}
are Tschirnhausen equivalent over $M$. 
By Lemma \ref{lemex} (i) and Theorem \ref{thiso} (1)
if $(2mn+3m+3n+18)(2mn+3m+3n-9)=0$ then $X^3-mX^2-(m+3)X-1$ and $X^3-nX^2-(n+3)X-1$ 
have the same splitting field over $M$ because $F_2(\ba,\bb;X)$ has multiple roots but 
also has a simple root. 

\begin{theorem}
For $m,n\in M$, two splitting fields of $X^3-mX^2-(m+3)X-1$ and of 
$X^3-nX^2-(n+3)X-1$ over $M$ coincide if and only if 
$F_2^+(m,n;X)F_2^-(m,n;X)$ has a root in $M$. 
\end{theorem}
\begin{example}
We take $M=\mathbb{Q}$. 
If $(m,n)\in\{(-1,5),$ $(-1,1259),$ $(0,54),$ $(5,1259)\}$ then 
$F_2^+(m,n;X)$ splits completely over $\mathbb{Q}$. 
If $(m,n)\in\{(-1,12),$ $(0,3),$ $(1,66),$ $(2,2389),$ $(3,54),$ $(5,12),$ 
$(12,1259)\}$ then $F_2^-(m,n;X)$ splits completely over $\mathbb{Q}$. 
Hence we see 
\[
L_{-1}=L_5=L_{12}=L_{1259},\quad L_0=L_3=L_{54},\quad L_1=L_{66},\quad L_2=L_{2389},
\]
where $L_m=\mathrm{Spl}_\mathbb{Q} (X^3-mX^2-(m+3)X-1)$. 
We have observed that for integers $m$ and $n$ in the range 
$-1\leq m< n\leq 100000$, $F_2^+(m,n;X)F_2^-(m,n;X)$ has a linear 
factor over $\mathbb{Q}$ only for the values of $(m,n)$ noted above. 
\end{example}

In the case of char $k\neq 2, 3$, we obtain by using (\ref{eqm}) that $F_2^+(m,n;X)$ and 
\[
g^+(m,n;X):=X^3+\frac{3(mn+6m-3n+9)}{2mn+3m+3n+18}X^2
-\frac{3(mn-3m+6n+9)}{2mn+3m+3n+18}X-1
\]
are Tschirnhausen equivalent over $M$. 
We also see that 
$F_2^-(m,n;X)$ and 
\[
g^-(m,n;X):=X^3+\frac{3(mn-3m-3n-18)}{2mn+3m+3n-9}X^2
-\frac{3(mn+6m+6n+9)}{2mn+3m+3n-9}X-1
\]
are Tschirnhausen equivalent over $M$. 

Putting $Z=(X-1)/(X+2)$ we have $X=-(2Z+1)/(Z-1)$ and 
\begin{align*}
h^+(m,n;Z)\, &:=\, \frac{1}{3^3(m-n)}\,g^+\Bigl(m,n;\frac{-(2Z+1)}{Z-1}\Bigr)\\
\, &=\, Z^3-\frac{mn+3n+9}{m-n}Z^2-\frac{mn+3m+9}{m-n}Z-1,\\
h^-(m,n;Z)\, &:=\, \frac{-1}{3^3(m+n+3)}\,g^-\Bigl(m,n;\frac{-(2Z+1)}{Z-1}\Bigr)\\
\, &=\, Z^3+\frac{mn+3m+3n}{m+n+3}Z^2+\frac{mn-9}{m+n+3}Z-1. 
\end{align*}
We also obtain 
\[
\mathrm{Disc}_Z(h^+(m,n;Z))=\frac{\Dela^2 \Delb^2}{(m-n)^4},\quad 
\mathrm{Disc}_Z(h^-(m,n;Z))=\frac{\Dela^2 \Delb^2}{(m+n+3)^4}. 
\]

In the case of char $k=3$, using the result in Subsection \ref{subsech3}, 
we directly see that $F_2^+(m,n;X)$ (resp. $F_2^-(m,n;X)$) and $h^+(m,n;Z)$ 
(resp. $h^-(m,n;Z)$) are Tschirnhausen equivalent over $M$. 
Therefore we get the following theorem which is an analogue to the results 
of Morton \cite{Mor94} and Chapman \cite{Cha96}. 
\begin{theorem}\label{thC3}
Assume that char $k\neq 2$. 
For $m,n\in M$, two splitting fields of $X^3-mX^2-(m+3)X-1$ and of $X^3-nX^2-(n+3)X-1$ 
over $M$ coincide if and only if there exists $z\in M$ such that either 
\begin{align*}
n\,=\,\frac{m(z^3-3z-1)-9z(z+1)}{mz(z+1)+z^3+3z^2-1}\ \mathit{or}\ 
n\,=\,-\frac{m(z^3+3z^2-1)+3(z^3-3z-1)}{mz(z+1)+z^3+3z^2-1}.
\end{align*}
\end{theorem}
\begin{proof}
We should check only the case of $(m-n)(m+n+3)(2mn+3m+3n+18)(2mn+3m+3n-9)=0$. 
If we take $z=0$ then we have $n=m$ or $n=-m-3$. 
If we take $z=1$ then we get $n=-3(m+6)/(2m+3)$ or $n=-3(m-3)/(2m+3)$ 
which corresponds to $2mn+3m+3n+18=0$ or $2mn+3m+3n-9$. 
\end{proof}
By applying Hilbert's irreducibility theorem and Siegel's  theorem for curves of 
genus $0$ (cf. \cite[Theorem 6.1]{Lan78}, \cite[Chapter 8, Section 5]{Lan83}, 
\cite[Theorem D.8.4]{HS00}) to Theorem \ref{thC3} respectively, 
we get the following corollaries: 
\begin{corollary}\label{Hil2}
If $M$ is a Hilbertian field then for a fixed $m\in M$ there exist infinitely many $n\in M$ 
such that two splitting fields of $X^3-mX^2-(m+3)X-1$ and of $X^3-nX^2-(n+3)X-1$ over $M$ 
coincide. 
\end{corollary}
\begin{corollary}\label{Siegel2}
Let $M$ be a number field and $\mathcal{O}_M$ the ring of integers in $M$. 
For a given integer $m\in \mathcal{O}_M$, there exist only finitely many integers 
$n\in\mathcal{O}_M$ such that two splitting fields of $X^3-mX^2-(m+3)X-1$ and of 
$X^3-nX^2-(n+3)X-1$ over $M$ coincide. 
\end{corollary}
\begin{proof}
We may apply Siegel's theorem to Theorem \ref{thC3} because the discriminant of 
the denominator $mz(z+1)+z^3+3z^2-1$ $(=z^3+(m+3)z^2+mz-1)$ of the equalities 
in Theorem \ref{thC3} is given as $(m^2+3m+9)^2$. 
\end{proof}

\section{Some sextic generic polynomials}

Assume that char $k\neq 3$. 
Let $H_1$ and $H_2$ be subgroups of $\mathfrak{S}_3$. 
Let $k(s,t)$ be the rational function field over $k$ with variables $s,t$. 
We take a $k$-generic polynomial $f_3(\ba;X)\in k(s)[X]$ for $H_1$ with 
one parameter $s$ and a $k$-generic polynomial $f_3(\bb;X)\in k(t)[X]$ 
for $H_2$ also with one parameter $t$ where $\ba\in k(s)^3$ and $\bb\in k(t)^3$. 
Assume that $g^{(H_1,H_2)}(s,t;X):=F_2(\ba,\bb;X)$ has no multiple root. 
Then, by Theorem \ref{thgen}, $g^{(H_1,H_2)}(s,t;X)$ is a $k$-generic polynomial for 
$H_1\times H_2$ with two parameters $s,t$. 
Note that there exists no $\mathbb{Q}$-generic polynomial with 
one parameter except for the groups $\{1\},C_2,C_3,\mathfrak{S}_3$ 
(cf. \cite{BR97}, \cite{Led07}, \cite{CHKZ}). 
If we take 
$(H_1,H_2)\in \{(\mathfrak{S}_3,\mathfrak{S}_3)$, 
$(\mathfrak{S}_3,C_3)$, $(\mathfrak{S}_3,C_2)$, $(\mathfrak{S}_3,\{1\})$, $(C_3,C_2)\}$, 
then it follows from Theorem \ref{th} that 
$g^{(H_1,H_2)}(s,t;X)$ is irreducible over $k(s,t)$ because we have $L_\ba\cap L_\bb=k(s,t)$. 
Hence $H_1\times H_2$ can be regarded as a transitive subgroup of $\mathfrak{S}_6$ naturally. 

~\\
(1) The case of $(H_1,H_2)=(\mathfrak{S}_3,\mathfrak{S}_3)$. 
We take $\ba=(0,s,-s), \bb=(0,t,-t)$. 
Then we have $f_3(\ba;X)=X^3+sX+s, f_3(\bb;X)=X^3+tX+t$, 
$\Aa^3\Bb^2-27\Ab^3\Da=-729s^2t^2(4st+27s+27t)$, and 
\begin{align*}
g^{(\mathfrak{S}_3,\mathfrak{S}_3)}(s,t;X)\, :=\, 
\frac{1}{3^6}\, &\, F_2(0,s,-s,0,t,-t;3X)\\
\, =\, X^6&+\frac{2t}{s(4s+27)}X^4+\frac{t}{s^2(4s+27)}X^3\\
&+\frac{t^2}{s^2(4s+27)^2}X^2+\frac{t^2}{s^3(4s+27)^2}X+\frac{(s-t)t^2}{s^4(4s+27)^3}; 
\end{align*}
this is a $k$-generic polynomial for 
$\mathfrak{S}_3\times \mathfrak{S}_3\subset \mathfrak{S}_6$. 

~\\
(2) The case of $(H_1,H_2)=(\mathfrak{S}_3,C_3)$. 
We take $\ba=(0,s,-s), \bb=(t,-t-3,1)$. 
Then we have $f_3(\ba;X)=X^3+sX+s, f_3(\bb;X)=X^3-tX^2-(t+3)X-1$, 
$\Aa^3\Bb^2-27\Ab^3\Da=729s^2(t^2+3t+9)^2(t^2+3t+9+s)$, and 
\begin{align*}
g^{(\mathfrak{S}_3,C_3)}(s,t;X)\, :=\, 
F_2(&0,s,-s,t,-t-3,1;X)\\
\,=\,X^6&-\frac{6(t^2+3t+9)}{s(4s+27)}X^4-\frac{(2t+3)(t^2+3t+9)}{s^2(4s+27)}X^3\\
&+\frac{9(t^2+3t+9)^2}{s^2(4s+27)^2}X^2+\frac{3(2t+3)(t^2+3t+9)^2}{s^3(4s+27)^2}X\\
&+\frac{(t^2+3t+9)^2(4st^2+27t^2+12st+9s+81t+243)}{s^4(4s+27)^3}; 
\end{align*}
this is a $k$-generic polynomial for $\mathfrak{S}_3\times C_3\cong C_3\wr C_2\cong 
(C_3\times C_3)\rtimes C_2\subset \mathfrak{S}_6$. 

After specializing $(s,t)\mapsto (a,b)\in M^2$, we see that if $b^2+3b+9+a=0$ then 
$X^3+aX+a$ and $X^3-bX^2-(b+3)X-1$ have the same splitting field over $M$. 
More precisely, 
\begin{align*}
X^3-(b^2+3b+9)X-(b^2+3b+9)\quad \mathrm{and}\quad X^3-bX^2-(b+3)X-1 
\end{align*}
are Tschirnhausen equivalent over $M$. 
We also see that if $4ab^2+27b^2+12ab+9a+81b+243=0$ then 
$X^3+aX+a$ and $X^3-bX^2-(b+3)X-1$ have the same splitting field over $M$. 
Hence we see 
\begin{align*}
X^3-\frac{27(b^2+3b+9)}{(2b+3)^2}X-\frac{27(b^2+3b+9)}{(2b+3)^2}\quad 
\mathrm{and}\quad X^3-bX^2-(b+3)X-1 
\end{align*}
are Tschirnhausen equivalent over $M$. Here the equivalence is given by 
an affine form.  

~\\
(3) The case of $(H_1,H_2)=(\mathfrak{S}_3,C_2)$. 
We take $\ba=(0,s,-s), \bb=(0,-t,0)$. 
Then we have $f_3(\ba;X)=X^3+sX+s, f_3(\bb;X)=X(X^2-t)$, 
$\Aa^3\Bb^2-27\Ab^3\Da=729s^2t^3(4s+27)$, and 
\begin{align*}
g^{(\mathfrak{S}_3,C_2)}(s,t;X)\, &:=\, \frac{1}{3^6}\,F_2(0,s,-s,0,-t,0;3X)\\
\,&\ =\,X^6-\frac{2t}{s(4s+27)}X^4+\frac{t^2}{s^2(4s+27)^2}X^2+\frac{t^3}{s^4(4s+27)^3}; 
\end{align*}
this is a $k$-generic polynomial for 
$\mathfrak{S}_3\times C_2\cong D_6\subset \mathfrak{S}_6$ 
where $D_6$ is the dihedral group of order $12$. 

~\\
(4) The case of $(H_1,H_2)=(\mathfrak{S}_3,\{1\})$. 
Assume that char $k\neq 2$. 
We take $\ba=(0,s,-s), \bb=(0,-1,0)$; then 
$f_3(\ba;X)=X^3+sX+s, f_3(\bb;X)=X(X+1)(X-1)$ and 
$\Aa^3\Bb^2-27\Ab^3\Da=729s^2(4s+27)$. 
We obtain
\begin{align*}
g^{(\mathfrak{S}_3,\{1\})}(s,t;X)\, &:=\, \frac{1}{3^6}\, F_2(0,s,-s,0,-1,0;3X)\\
\,&\ =\,X^6-\frac{2}{s(4s+27)}X^4+\frac{1}{s^2(4s+27)^2}X^2+\frac{1}{s^4(4s+27)^3}; 
\end{align*}
this is a $k$-generic polynomial for $\mathfrak{S}_3\subset\mathfrak{S}_6$. 
We also obtain the following $k$-generic polynomial for $\mathfrak{S}_3$ 
with one parameter $s$: 
\begin{align*}
h^{\mathfrak{S}_3}(s;X)\,&:=\,(s(4s+27))^6\, 
g^{(\mathfrak{S}_3,\{1\})}\Bigl(s,t;\frac{X}{s(4s+27)}\Bigr)\\
\, &\ =\,X^6-2s(4s+27)X^4+s^2(4s+27)^2X^2+s^2(4s+27)^3. 
\end{align*}
Two polynomials $f_3(\ba;X)=X^3+sX+s$ and $h^{\mathfrak{S}_3}(s;X)$ 
have the same splitting field over $k(s)$. 

~\\
(5) The case of $(H_1,H_2)=(C_3,C_2)$. 
We take $\ba=(s,-s-3,1), \bb=(0,-t,0)$. 
Then we have $f_3(\ba;X)=X^3-sX^2-(s+3)X-1, f_3(\bb;X)=X(X^2-t)$, 
$\Aa^3\Bb^2-27\Ab^3\Da=-729t^3(s^2+3s+9)^2$, and 
\begin{align*}
g^{(C_3,C_2)}(s,t;X)\, &:=\, F_2(s,-s-3,1,0,-t,0;X)\\
\,&\ =\,X^6-\frac{6t}{s^2+3s+9}X^4+\frac{9t^2}{(s^2+3s+9)^2}X^2
-\frac{(2s+3)^2t^3}{(s^2+3s+9)^4}; 
\end{align*}
this is a $k$-generic polynomial for $C_3\times C_2\cong C_6\subset\mathfrak{S}_6$. \\

In the case of char $k=3$, we should use $F_0(\ba,\bb;X)$ instead of $F_2(\ba,\bb;X)$ 
because of the result in Subsection \ref{subsech3}. 
Here we give only $g^{(H_1,H_2)}(s,t;X):=F_0(\ba,\bb;X)$ 
for each $(H_1,H_2)$ in the case of char $k=3$: 
\begin{align*}
g^{(\frak{S}_3,\frak{S}_3)}(1/s,t;X)\,&=\,X^6-tX^4-tX^3+t^2X^2-t^2X-t^2(st-1),\\
g^{(\frak{S}_3,C_3)}(1/s,t;X)\,&=\,X^6+tX^5+t(t+1)X^4+(st^3-t^2+1)X^3\\
&\hspace*{5mm}-t(st^3-t+1)X^2-t(st^3+1)X+s^2t^6+st^4-st^3+1,\\
g^{(\frak{S}_3,C_2)}(1/s,t;X)\,&=\,X^6-tX^4+t^2X^2+st^3,\\
g^{(\frak{S}_3,\{1\})}(1/s,t;X)\,&=\,X^6-X^4+X^2+s,\\
g^{(C_3,C_2)}(1/s,t;X)\,&=\,X^6+\frac{t^2}{s^4+s^2+1}X^2+\frac{s^2t^3}{s^8+s^4+1}.\\
\end{align*}

\begin{acknowledgment}
We would like to thank the referee for suggesting many improvements and corrections 
to the manuscript. 
We also thank him/her for providing an improved statement and an alternative proof 
of Lemma \ref{lemnoMult}. 
\end{acknowledgment}


\vspace*{1cm}
{\small
\hspace*{-0.6cm}
\begin{tabular}{ll}
Akinari HOSHI & Katsuya MIYAKE\\
Department of Mathematics & Department of Mathematics\\
Rikkyo University  & School of Fundamental Science and Engineering\\
3--34--1 Nishi-Ikebukuro Toshima-ku & Waseda University\\
Tokyo, 171-8501, Japan & 3--4--1 Ohkubo Shinjuku-ku\\
E-mail: \texttt{hoshi@rikkyo.ac.jp} & Tokyo, 169--8555, Japan\\
& E-mail: \texttt{miyakek@aoni.waseda.jp}
\end{tabular}
}

\end{document}